\documentclass[10pt]{amsart}
\usepackage[margin=1.4in]{geometry}
\usepackage{amssymb,comment}
\usepackage[shortlabels]{enumitem}
\usepackage{latexsym}
\usepackage{graphicx}
\usepackage{mathrsfs}
\usepackage{pdfsync}
\usepackage{subfigure}
\usepackage{amsmath,color}
\usepackage{tikz}
\usepackage[linktocpage=true]{hyperref}
\usepackage{tikz-cd} 

\hypersetup{ colorlinks   = true, 
urlcolor  = blue, 
linkcolor    = black, 
citecolor   = black 
}

\def\l{{\lambda}}

\def\e{{\varepsilon}}
\def\beq{\begin{equation}}
\def\eeq{\end{equation}}

\newcommand{\Z}{{\mathbb Z}}
\newcommand{\R}{{\mathbb R}}

\newcommand{\C}{{\mathbb C}}

\newcommand{\MM}{{\mathrm M}}

\newcommand{\PP}{{\mathbb P}}

\newcommand{\CD}{{\mathcal D}}

\newcommand{\CS}{{\mathcal S}}

\newcommand{\tr}{\mathrm{tr} }
\newcommand{\SVG}{\mathrm{(SVG)} }
\newcommand{\FI}{\mathrm{(FI)} }

\newtheorem{theorem}{Theorem}
\newtheorem{remark}{Remark}
\newtheorem{lemma}{Lemma}
\newtheorem{defi}{Definition}
\newtheorem{prop}{Proposition}

\newtheorem{corollary}{Corollary}
\newtheorem{example}{Example}
\begin{document}

\title[Equivalent Conditions for Domination]{Equivalent Conditions for Domination of $\mathrm{M}(2,\C)$-sequences}
\author{Chang Sun and Zhenghe Zhang}

\address{Department of Mathematics, University of California, Riverside, CA-92521, USA}
\email{chang.sun@email.ucr.edu}

\address{Department of Mathematics, University of California, Riverside, CA-92521, USA}
\email{zhenghe.zhang@ucr.edu}

\begin{abstract}
	It is well known that a $\mathrm{SL}(2,\C)$-sequence is uniformly hyperbolic if and only it satisfies a uniform exponential growth condition. Similarly, for  $\mathrm{GL}(2,\C)$-sequences whose determinants are uniformly bounded away from zero, it has dominated splitting if and only if it satisfies a uniform exponential gap condition between the two singular values. Inspired by \cite{quas}, we provide a similar equivalent description in terms of singular values for $\mathrm{M}(2,\C)$-sequences that admit dominated splitting. We also prove a version of the Avalanche Principle for such sequences.
\end{abstract}

\maketitle
\tableofcontents

\section{Introduction}

Throughout this paper, we consider $B\in\ell^\infty(\Z, \MM(2,\C)$. In particular, we fix a large $M>0$ so that $\sup_{j\in\Z}\|B(j)\|<M$, where $\|A\|$ always denotes the operator norm for $A\in\MM(2,\C)$. Throughout this paper, unless otherwise stated, all matrices are assumed to be nonzero and to have operator norm bounded above by $M$. We also let $C>0, c>0$ be universal constants, where $C$ is large and $c$ is small. We define
\begin{equation*}\label{eq:cocycleiteration}
B_n(j)=\begin{cases}B(j+n-1)\cdots B(j), & n\ge1,\\ I_2 , & n=0,
\end{cases}
\end{equation*}
where $I_2$ is the identity matrix and
$$
B_{-n}(j)= [B_{n}(j-n)]^{-1}=B(j-n)^{-1}\cdots B(j-1)^{-1}, \ n\ge 1,
$$
if all matrices involved are invertible. Regarding $\mathrm{M}(2,\C)$-sequences admitting dominated splitting, the following definition is introduced in \cite{alkornzhang} and is proved to be appropriate. 
	\begin{defi}\label{d:domination}
	We say that $B:\Z\to\mathrm{M}(2,\C)$ admits if for each $j\in\Z$, there are one-dimensional spaces $E^u(j)$ and $E^s(j)$ of $\C^2$ with the following properties:
	
	\begin{enumerate}[(a)]
		\item $E^u,E^s$ are $B$--invariant in the sense that for all $j\in\Z$, it holds that
		$$
		B(j)[E^u(j)] \subseteq E^u(j+1) \mbox{ and } B(j)[E^s(j)]\subseteq E^s(j+1).
		$$
		\item There exist $N \in \Z_{+}$ and $\lambda>1$ such that 
		$$
		\|B_{N}(j)\vec u(j)\|>\lambda \|B_N(j)\vec s(j)\|
		$$
		for all $j\in\Z$ and all unit vectors $\vec u(j)\in E^u(j)$ and $\vec s(j)\in E^s(j)$.
		\item $\inf_{j\in\Z}d\big(E^u(j), E^s(j)\big)> 0$.
		\item  For the same $N$ from (2), it holds that $\inf_{j\in\Z}\|B_N(j)\|>0$.
	\end{enumerate} 
\end{defi}

A similar definition may be found at \cite{quas} for sequence of bounded linear operators on a Banach space. Condition (b) or (c) implies that $E^s(j)\neq E^u(j)$, hence $\C^2=E^s(j)\oplus E^u(j)$ for all $j\in\Z$. In condition (c), $d(W,V)$ denotes a distance between two one-dimensional spaces $W$ and $V$ of $\C^2$. For its definition, see equation \eqref{eq:DistofComplexLine} at the beginning of Section~\ref{s:obtainDomination}. Here we say the space $E^u$ dominates the space $E^s$. As noted in \cite{alkornzhang}, under conditions (a)-(c), condition (d) actually implies something stronger:
\[
\inf_{j\in\Z}\|B_n(j)\|>0\mbox{ for all } n\in\Z_+.\label{eq:equiv_c4} \tag*{(d)'}
\]
 We shall sometimes use or prove condition (d)'. From now on, $B \in \CD\CS$ means $B$ has dominated splitting. 

We have the following equivalent conditions for $\CD\CS$-sequences. For any nonzero matrix $A\in\MM(2,\C)$, we define $\sigma_1(A)=\|A\|$ and $\sigma_2(A)=\frac{|\det(A)|}{\|A\|}$ which are the so-called singular values of $A$. One may find more detailed information about the singular values of $A$ at the beginning of Section~\ref{s:obtainDomination}. Let $\mu>1$ and $0<\e<\mu$ where $\e$ may be arbitrarily small. Then we define:
\begin{align*}
&\mbox{(SVG)}\hskip 0.5cm\sup_{j\in\Z}\left\{\frac{\sigma_2(B_n(j))}{\sigma_1(B_{n+1}(j))},\  \frac{\sigma_2(B_n(j+1))}{\sigma_1(B_{n+1}(j))}\right\}<C\mu^{-n} \mbox{ for all }n\ge 0.\\
&\mbox{(FI)}\hskip 0.8cm \sup_{j\in\Z}\left\{\frac{\sigma_1(B_n(j))}{\sigma_1(B_{n+1}(j))},\  \frac{\sigma_1(B_n(j+1))}{\sigma_1(B_{n+1}(j))}\right\}<C\mu^{(1-\e)n} \mbox{ for all }n\ge 1,
\end{align*}
where following the terminology of \cite{quas} $\SVG$ stands for singular value gap and $\FI$ stands for fast invertibility. Our main goal is to prove the following theorem: 

\begin{theorem}\label{t:main}
Let  $B\in\ell^{\infty}(\Z,\mathrm{M}(2,\C))$. $B \in \CD\CS$ if and only if $B$ satisfies both $\mathrm{(SVG)}$ and $\mathrm{(FI)}$.
	\end{theorem}
	
	 Domination is  an intensively studied notion in dynamical systems which naturally generalizes the notion of uniform hyperbolicity. It was introduced by and plays a key role in the works of Ma\~ n\'e \cite{mane}  and Liao \cite{liao} on Smale' s stability conjecture. The term \textit{dominated splitting} was introduced by Ma\~n\'e in \cite{mane}. We refer the readrs to \cite{bochiviana, bonatti, pujals} and the referenes therein regarding recent works in smooth dynamical systems involving domination.
	
	 In this paper, we focus on equivalent conditions for domination via singular values. For bounded $\mathrm{SL}(2,\C)$-sequences, we can use conditions (a) and (b) to define uniform hyperbolicity. It is straightforward to see that for such sequences, $\mathrm{(SVG)}$ is equivalent to the uniform exponential growth $\mathrm{(UEG)}$ condition 
	\[
	\inf_{j\in\Z}\|B_n(j)\|\ge C\mu^{\frac n2} \mbox{ for all } n\ge 1.
	\]
	It is well-known that uniform hyperbolicty of such sequences is equivalent to $\mathrm{(UEG)}$, see e.g. \cite{yoccoz} or \cite{zhang2} for detailed information. 
	
	In the context of bounded $\mathrm{GL}(2,\C)$-sequences whose determinants are unifromly bounded away from zero, one only needs conditions (a) and (b) of Definition~\ref{d:domination} to define domination. For such sequences, domination is equivalent to $\mathrm{(SVG)}$, which is easily seen to be equivalent to a weaker form as:
	\begin{equation}\label{eq:wSVG}
	\sup_{j\in\Z}\frac{\sigma_2(B_n(j))}{\sigma_1(B_{n}(j))}<C\mu^{-n}, \mbox{ for all } n\ge 1.
	\end{equation}
	Indeed, for such sequences, one can convert them into $\mathrm{SL}(2,\C)$-sequences by considering 
	\[
	A(j)=\frac1{\sqrt{\det(B(j))}}B(j).
	\] 
	Then all claimed results above follows from those of $\mathrm{SL}(2,\C)$-sequences.  One may also see \cite{bochigourmelon} for related discussion.
	
	In the context of general bounded $\mathrm{GL}(2,\C)$-sequences, $\SVG$ is strictly stronger than \eqref{eq:wSVG}. Indeed, by choosing $n=0$, one can easily see that $\SVG$ implies $\inf_{j\in\Z}\|B(j)\|>0$ (see e.g. the proof of Lemma~\ref{l:uLargeNorm}) while \eqref{eq:wSVG} doesn't gurantee that. However, even for $\mathrm{GL}(2,\C)$-sequences satisfying $\inf_{j\in\Z}\|B(j)\|>0$, $\SVG$ is still strictly stronger than  \eqref{eq:wSVG}. This is because we can let $|\det(B(j))|$ to be arbitrarily small. For instance, by choosing $B_n(j)$ and $B(j+n)$ appropriately, one can make $\frac{\sigma_2(B_n(j))}{\sigma_1(B_{n+1}(j))}$ arbitrarily large  while keeping $\frac{\sigma_2(B_{n+1}(j))}{\sigma_1(B_{n+1}(j))}$ arbitrarily small. $\SVG$ was first introduced in \cite{blumenthalmorris} where it is used to show domination for injective operator-valued cocycles, namely $B(j)$'s are allowed to be injective bounded linear operators on a Banach space. In fact, even in the context of $\MM(2,\C)$, $\SVG$ arises naturally in showing the existence of the invariant directions $E^u$ and $E^s$, see the proof of Lemma~\ref{l:SVGImpliesExistence}.
	
	While $\SVG$ is sufficient to show existence of $E^u$ and $E^s$, it is however not sufficent to gurantee $E^s\neq E^u$. That is the place where $\FI$ is used. $\FI$ was first introduced in \cite{quas}, which if translated into the context of Theorem~\ref{t:main} can be rewritten as:
	\begin{equation}\label{eq:sFI}
	\sup_{j\in\Z, n\ge 1}\frac{\sigma_1(B_n(j))}{\sigma_1(B_{n+1}(j))}<C.
	\end{equation}
	Similar to \cite{blumenthalmorris}, \cite{quas} considers domination for sequences of bounded linear operators on a Banach space. $\FI$ is introduced to avoid the injectiveness assumption of \cite{blumenthalmorris}. We wish to point out that even in the context of general $\mathrm{GL}(2,\C)$-sequences (hence, one has injectiveness), $\SVG$ is not suffcient to show $E^u\neq E^s$, see e.g. Example~\ref{example} and its remark. Indeed, \cite{blumenthalmorris} considered cocycles defined over compact toplogical spaces where injectiveness implies 
	uniform lower bound of $|\det(B(j))|$ in the context of Theorem~\ref{t:main}. Compared with Theorem~\ref{t:main}, the main result of \cite{quas} is stronger in the sense that they considered more general sequences, obtained explicit lower bounds on $d(E^s, E^u)$, and adopted a more general version of $\SVG$ which allows them to remove the uniform boundedness assumption on $\|B(j)\|$. 
	
	However, while it is in the simplest context, bounded $\MM(2,\C)$-sequences still keep the essential difficulty. Thus it is worthwhile to explore domination in the simplest nontrivial scenario as it allows one to see the dynamics relatively clearly. Compared with \cite{quas}, one of the main advantages of approach is that the proof of Theorem~\ref{t:main} is significantly simpler, which is based on a soft argument used in \cite{zhang2}. Basically, instead of obtaining explict lower bound of $d(E^s(j), E^u(j))$, we prove $E^s(j)\neq E^u(j)$ in the singular case and reduce another unresolved case into this case via embedding the sequence into its hull. This soft argument also allows us to weaken the condition \eqref{eq:sFI} to our $\FI$. In fact, by our proof one can see that to get $\inf_{j\in\Z}d(E^s(j),E^u(j))>0$, the real competition lies between $\frac{\sigma_2(B_n(j))}{\sigma_1(B_n(j+1))}$ and $\frac{\sigma_1(B_n(j))}{\sigma_1(B_n(j+1))}$. We really want the decay rate of the former beats the growth rate of the latter. In some sense, we have rigidity of the growth rate of $\frac{\sigma_1(B_n(j))}{\sigma_1(B_n(j+1))}$: as long as it grows slower than the decaying rate of $\frac{\sigma_2(B_n(j))}{\sigma_1(B_n(j+1))}$ (which is of exponential rate by $\SVG$), it is uniformly bounded above.
	 One may see the proof of Lemmas~\ref{l:EsNeqEu_singularCase} and ~\ref{l:EsAwayFromEu} for more detailed information. 
	
	On the other hand, even the stronger version of $\FI$ as stated in \eqref{eq:sFI} is automatically satisfied for bounded $\mathrm{GL}(2,\C)$-sequences with determinants unifromly bounded away from zero. This is due to the following fact:
	\[
	\sigma_1(B_n(j+1))\ge \sigma_1(B_n(j))\sigma_2(B(j+n))=\sigma_1(B_n(j))\frac{|\det(B(j+n))|}{\|B(j+n)\|}\ge c\sigma_1(B_n(j)).
	\]
	
    In summary, $\FI$ arises naturally if one wants to separate $E^s$ and $E^u$ in cases one does not have uniform invertibility of the operators.
    
We also further explore our techinques and prove the following version of the so-called \emph{Avalanche Principle} for $\MM(2,\C)$-sequences. It generalizes \cite[Theorem 5]{zhang2} and consequently, extends all prior versions of the Avalanche Principle for sequences of $2\times 2$ matrices.

\begin{theorem}\label{t:APSingular}
	
	\label{l.avlanche-principle}
	Let $B(j):\Z\to\mathrm{M}(2,\C)$ be that $0<c<\|B(j)\|<M$ for all $j\in\Z$. Suppose there is a $\mu>1$ large such that for each $j$, it holds that:
	\begin{align}\label{condition-AP3}
		&\frac{\sigma_2(B(j))}{\sigma_1(B(j))}\le \mu^{-1}; \\ \label{condition-AP4}
		&\frac{\sigma_1(B(j+1))\sigma_1(B(j))}{\sigma_1(B(j+1)B(j))}\le \mu^{\frac14}.
	\end{align}
	Then $B$ has dominated splitting and it holds for each $j\in\Z$ and each $n\ge 3$ that
	\begin{equation}\label{eq:AP}
		\left|\log\|B_n(j)\|+\sum_{k=1}^{n-2}\log\|B(j+k)\|-\sum_{k=0}^{n-2}\log\|B(j+k+1)B(j+k)\|\right|\le Cn\mu^{-\frac12}.
	\end{equation}
\end{theorem}

Avalanche Principle is first introduced in \cite{goldstein} for finite $\mathrm{SL}(2,\R)$ sequences. Together with large deviation type of estimates for the associated Lyapuonv exponent, it is proved to be a powerful tool in establishing quantative continuity of the Lyapunov exponent and the integrated density of states of the associated ergodic Schr\"odinger operators.
There are numerous generalizations since the original work of Goldstein-Schlag, e.g. orders of the matrices have been generalized from $2$ to any $d\ge 2$, real valued matrices to complex valued ones. We refer the readers to the work of Bourgain-Jitomirskaya \cite[Lemma 5]{bourgainjitomirskaya}, Bourgain \cite[Lemma 2.6]{bourgain}, and Schlag \cite[Lemma 1]{schlag} for of invertible matrices, and Duarte-Klein \cite[Section 2.4]{duarteklein} for  general $\MM(d,\R)$-matrices. 

However, it was first observed in \cite{zhang2} that there is a close relation between uniformly hyperbolic sequence of matrices and the Avalanche Principle in the scenario of $\mathrm{SL}(2,\C)$-sequences. As a consequence, a relatively short and dynamical proof of the Avalanche Principle was obtained in \cite{zhang2}. Likewise, no relation between domination and the Avalanche  Principle for  squences of possibly non-invertible matrices seemed to be explored (in fact, as far as we know, the work of \cite{duarteklein} is the only place where non-invertible finite sequences were explored). One of the main purpose of Theorem~\ref{t:APSingular} is to explore such a relation and provide a relatively short and dynamical proof for $\MM(2,\C)$-sequences. While it is possible for us to consider general $\MM(d,\C)$ sequences, we again wish to do things in the simplest nontrivial setup to emphasize the main ideas and key tools behind the proof. The proof is a combination of the techniques we used to prove Theorem~\ref{t:main} and outline of \cite[Section 4]{zhang2}. Theorem~\ref{t:APSingular} has the  potential to be highly useful in establishing  H\"older continuity of the Lyapunov exponent for $\MM(2,\C)$-cocycles.

	To conclude this section, we further explore the following example from \cite{alkornzhang}. It provides an example in the scenario of bounded $\mathrm{GL(2,\R)}$-sequences whose norm is uniformly bounded from below where $\SVG$ holds true and $\FI$ fails. It also shows that merely $\SVG$ is not enough to guarantee the separation of $E^s$ and $E^u$.

\begin{example}\label{example}

Define $\Lambda(j)=\left(\begin{smallmatrix}2^{2-|j|} & 0 \\ 0 & 2^{-|j|}\end{smallmatrix}\right), D(j)=\left(\begin{smallmatrix}1 & 1 \\ 0  &\   2^{-|j|}\end{smallmatrix}\right)$,
and 
\[
B(j):=D(j+1) \Lambda(j) D(j)^{-1}=\begin{pmatrix}2^{2-|j|} & -3 \\ 0 & 2^{-|j+1|}\end{pmatrix}.
\]
Claim: $B$ satisfies conditions (1), (2), and (4) of Definition~\ref{d:domination} as well as $\SVG$. But $B$ does not satisfies (3) or $\FI$.
\begin{proof}
By construction, $D$ converts the two obivous invariant line section of $\Lambda$ to those of $B$, which are:
\[
E^u(j)=\mathrm{span}\left\{\binom{1}{0}\right\}\mbox{ and } E^s(j)=\mathrm{span}\left\{\binom{1}{2^{-|j|}}\right\}.
\]
One readily checks that for all $j\in\Z$ and all unit vectors $\vec u(j)\in E^u(j)$ and $\vec s(j)\in E^s(j)$:
\[
\|B(j)\vec u(j)\|\ge 2 \|B(j)\vec s(j)\|.
\]
Hence, conditions (a) and (b) of Definition~\ref{d:domination} are satisfied where in (b) we may choose $N=1$. Moreover, it is clear that $\|B(j)\|\ge 3$ so that condition (d) is satisfied. However, by equation \eqref{eq:DistofComplexLine} we have
\[
d(E^u(j), E^s(j))=2|\det(\vec u(j),\vec s(j))|\le 2^{1-|j|}\to 0.
\]
Hence, condtion (c) is not satisfied. 

Next, let us explore $\SVG$ and $\FI$. It is straightforward to see that for all $n\ge 2$:
\begin{align*}
    B_{n}(j)&=D(j+n) \Lambda_{n}(j)D(j)^{-1}\\
    &=\begin{pmatrix}1 & 1 \\ 0 & 2^{-|j+n|}\end{pmatrix}\cdot \begin{pmatrix}2^{2n-\sum^{n-1}_{k=0}|j+k|} & 0 \\ 0 & 2^{-\sum^{n-1}_{k=0}|j+k|}\end{pmatrix}\cdot\begin{pmatrix}1 & -2^{|j|}\\ 0 & 2^{|j|}\end{pmatrix}\\
    &=2^{-\sum^{n-1}_{k=1}|j+k|}\begin{pmatrix}2^{2n-|j|} & -2^{2n}+1 \\ 0 & 2^{-|j+n|}\end{pmatrix}.
\end{align*}
Let $\|A\|_{\max }=\max _{i, j}\left|a_{i j}\right|$. It is a straightforward calculation to see that for all $A\in\MM(2,\C)$:
$$
\|A\|_ {\max }\leq\|A\| \leq 2\|A\|_{\max } 
$$ 
Then for all $j \in \Z$ and $n\geq 2$, we get 
\begin{align}\label{eq:sigma2_example1}
\nonumber \sigma_2(B_{n}(j))
&=\frac{|\det(B_{n}(j))|}{\|B_{n}(j))\|}\le\frac{|\det(B_{n}(j))|}{\|B_{n}(j))\|_{\max}} \\
&\le \frac{2^{-2\sum^{n-1}_{k=1}|j+k|}\cdot 2^{2n-|j|-|j+n|}}{2^{2n-\sum^{n-1}_{k=1}|j+k|}}\\
\nonumber &=2^{-\sum^{n}_{k=0}|j+k|}
\end{align} 
and 
\begin{equation}\label{eq:sigma1_example1}
\sigma_1(B_{n}(j))=\|B_{n}(j))\|\geq \|B_{n}(j))\|_ {\max }\ge 2^{-\sum^{n-1}_{k=1}|j+k|}\cdot (2^{2n}-1),
\end{equation}
which implies for all $n\ge 2$: 
 \begin{align*}
 &\sup_{j\in\Z}\left\{\frac{\sigma_2(B_n(j))}{\sigma_1(B_{n+1}(j))},\ \frac{\sigma_2(B_n(j+1))}{\sigma_1(B_{n+1}(j))} \right\}\\
& \le \sup\left\{\frac{2^{-\sum^{n}_{k=0}|j+k|}}{2^{-\sum^{n}_{k=1}|j+k|}\cdot (2^{2(n+1)}-1)},\ \  \frac{2^{-\sum^{n+1}_{k=1}|j+k|}}{2^{-\sum^{n}_{k=1}|j+k|}\cdot (2^{2(n+1)}-1)}\right\}\\
&\le \frac{1}{2^{2(n+1)}-1}\\
&<\frac{1}{2^{2n}}.
 \end{align*}
For $n=1$, we note that $\sigma_2(B(j))<2^{1-|j|-|j+1|}$ for all $j\in\Z$, which together with \eqref{eq:sigma1_example1} implies:

\begin{align*}
&\sup_{j\in\Z}\left\{\frac{\sigma_2(B(j))}{\sigma_1(B_2(j))},\frac{\sigma_2(B(j+1))}{\sigma_1(B_2(j))} \right\}\\
&\le \sup\left\{\frac{2^{1-|j|-|j+1|}}{2^{-|j+1|}\cdot (2^4-1)},\ \frac{2^{1-|j+1|-|j+2|}}{2^{-|j+1|}\cdot (2^4-1)}\right\}\\
&<\frac{2}{15}.
\end{align*}
For $n=0$, it is clear that we have $\sigma_1(B(j))=\|B(j)\|\ge 3$ which implies
$$
\frac{\sigma_2(B_0(j))}{\sigma_1(B(j))}=\frac{\sigma_2(I_2)}{\sigma_1(B(j))}=\frac{1}{\|B(j)\|}<\frac{1}{3}.
$$
To sum up,  we get for all $n\ge 0$:
$$
\sup_{j\in\Z}\left\{\frac{\sigma_2(B_n(j))}{\sigma_1(B_{n+1}(j))},\ \frac{\sigma_2(B_n(j+1))}{\sigma_1(B_{n+1}(j))} \right\}<\frac{1}{2^{2n}},
$$
which is nothing other than $\mbox{(SVG)}$.

Finally, $\|B_n(j)\|\le 2\|B_n(j)\|_{\max} \le 2^{-\sum^{n-1}_{k=1}|j+k|}\cdot 2^{2n+1}$ and \eqref{eq:sigma1_example1} imply for each $n\ge 2$ that:
\begin{align*}
	\frac{\|B_{n+1}(j)\|}{\|B_n(j+1)\|}&\le \frac{2^{-\sum^{n}_{k=1}|j+k|}\cdot 2^{2n+3}}{ 2^{-\sum^{n}_{k=2}|j+k|}\cdot (2^{2n}-1)}\\
	&=2^{-|j+1|}\cdot \frac{2^{2n+3}}{2^{2n}-1}\\
	&\to 0 \mbox{ as } j\to \pm\infty,
	\end{align*} 
hence, $\FI$ fails.
\end{proof}
\end{example}
\begin{remark}
	While Example~\ref{example} does not have $E^s\neq E^u$, we can embed the sequence $B$ to its hull and find a sequence in its hull which satisfies $\SVG$ while $E^s(j)=E^u(j)$. One may see Section~\ref{s:obtainDomination}, especially the proof of Proposition~\ref{p:seq_to_cocycle_svgFI} and Lemma~\ref{l:EsAwayFromEu} for more detailed information.
	\end{remark}

\section{$\C\PP^1$ and Singular Value Decomposition}
We first collect some facts about $\C\PP^1$ from \cite{alkornzhang}, where $\mathbb{C P}^1=\C \cup\{\infty\}$ is the one-dimensional complex projective space, or the Riemann sphere. We mainly use the following projection maps from $\mathbb{C}^2 \backslash\{\overrightarrow{0}\} \rightarrow \mathbb{C P}^1$ :
$$
\pi: \mathbb{C}^2 \backslash\{\overrightarrow{0}\} \rightarrow \mathbb{C P}^1 \text { where } \pi\begin{pmatrix} z_1\\z_2\end{pmatrix}=\frac{z_2}{z_1} .
$$

Through this projection, each one-dimensional space in $\C^2$ can be identified with a point in $\mathbb{C P}^1$. Hence, we may view a point $z \in \mathbb{C P}^1$ as the one-dimension space $\operatorname{span}\left\{\binom{1}{z}\right\}$ of $\C^2$. Note $\infty$ is considered to be $\operatorname{span}\left\{\binom{0}{1}\right\}$. For instance, by $\vec{v} \in z$, we mean $\vec{v}$ is a vector in the one dimensional space $z$. In particular, we let $\vec{z}$ denotes a unit vector in $z$. For any non-vector $\vec v$, we denote by $\vec v^\perp$ a unit vector that is orthogonal to $\vec v$. We use the following metric on $\mathbb{C P}^1$ :
\begin{equation}\label{eq:DistonCP1}
	d\left(z, z^{\prime}\right)= \begin{cases}\frac{2\left|z-z^{\prime}\right|}{\sqrt{\left(1+|z|^2\right)\left(1+\left|z^{\prime}\right|^2\right)}}, & z, z^{\prime} \in \C, \\ \frac{2}{\sqrt{1+|z|^2}}, & z^{\prime}=\infty.\end{cases}
\end{equation}

Let $\vec{u}$ and $\vec{v}$ be two nonzero vectors in $\mathbb{C}^2$. We let $\left(\vec{u}, \vec{v}\right) \in \mathrm{M}(2, \mathbb{C})$ denotes the matrix whose column vectors are $\vec{u}$ and $\vec{v}$. Then a direct computation shows that
\begin{equation}\label{eq:Dist_to_Det}
	d\left(\pi(\vec{u}), \pi\left(\vec{v}\right)\right)=\frac{2\left|\operatorname{det}\left(\vec{u}, \vec{v}\right)\right|}{\|\vec{u}\| \cdot\left\|\vec{v}\right\|}
\end{equation}
In particular, if $\vec{u}$ and $\vec{v}$ are two unit vectors, then we have
\begin{equation}\label{eq:Dist_to_Det2}
	d\left(\pi(\vec{u}), \pi\left(\vec{v}\right)\right)=2\left|\operatorname{det}\left(\vec{u}, \vec{v}\right)\right|,
\end{equation}
which clearly implies that
\begin{equation}\label{eq:orth_prev_dist}
	d\left(\pi(\vec{u}^\perp), \pi\left(\vec{v}^\perp\right)\right)=d\left(\pi(\vec{u}), \pi\left(\vec{v}\right)\right).
\end{equation}
In other words, if we define $\mathscr O(z)=(z)^\perp: \C\PP^1\to\C\PP^1$ to be $z^\perp:=\pi(\vec z^\perp)$, then $\mathscr O$ preserves the distance $d$.
Since one dimensional space can be identified by the points in $\mathbb{C P}^1$, abusing the notation slightly, for two one-dimensional subspaces $V$ and $W$ of $\mathbb{C}^2$, we define

\begin{equation}\label{eq:DistofComplexLine}
	d(U, V):=d(\pi(\vec{u}), \pi(\vec{v}))
\end{equation}
where $\vec{u} \in U$ and $\vec{v} \in V$ are nonzero vectors.

Let $A=\begin{pmatrix}a & b \\ c & d\end{pmatrix}\in \mathrm{M}(2, \mathbb{C})$ be a nonzero matrix. Under the projection $\pi$, there is an induced projectivized map of $A$ acting on projective space $\left(\mathbb{C P}^1\right) \backslash\{\alpha\}$, where $\alpha$ is the eigenspace of the 0 eigenvalue of $A$, if such exists. We denote the induced map by $A \cdot z$. Then a direct computation shows that

$$
A \cdot z:\left(\mathbb{C P}^1\right) \backslash\{\alpha\} \rightarrow \mathbb{C P}^1, A \cdot z=\frac{c+d z}{a+b z}
$$
Recall $\vec{z}$ denotes a unit vector in the one-dimensional space $z \in \mathbb{C P}^1$. Then it holds that 
\begin{equation}\label{eq:matrix_change_d}
	d\left(A \cdot z, A \cdot z^{\prime}\right)  =\frac{|\operatorname{det}(A)|}{\|A \vec{z}\| \cdot\left\|A \vec{z}^{\prime}\right\|}d\left(z, z^{\prime}\right).
\end{equation}
In particular, for all $U\in\mathrm{U}(2)$, the set of all $2\times2$ unitary matrices, it holds that
\begin{equation}\label{eq:Unitary_pre_d}
	d(U\cdot z,U\cdot  z')=d(z, z')\mbox{ for all }z, z'\in\C\PP^1.
\end{equation}

Next we collect some facts about singular value decomposition for $A\in\mathrm M(2,\C)$. It is a standard fact that we can decompose $A$ as $A=U\Lambda V^*$, where $U, V\in\mathrm{U}(2)$ and $\Lambda=\left(\begin{smallmatrix}\sigma_1(A)& 0\\ 0 & \sigma_2(A)\end{smallmatrix}\right)$. Moreover, we have
\begin{equation}\label{eq:SingularValues}
	\sigma_1(A)=\|A\|=\sup_{\vec \|v\|=1}\|A\vec v\|,\ \sigma_2(A)=\frac{|\det(A)|}{\|A\|}=\inf_{\vec \|v\|=1}\|A\vec v\|,
	\end{equation}
	 which are so-called singular values of $A$. Moreover, column vectors of $V$ are eigenvectors of $A^*A$ where the first column vector corresponds to the most expanding direction of $A$ and the second corresponds to the most contracted direction. Notice that two directions are orthogonal. Similarly, column vectors of $U$ are eigenvectors of $AA^*$, where the first column vector corresponds to the most expanding direction of $A^*$ and the second corresponds to the most contracted direction. 

With some fixed choice of $V$, we can view  $U$, $V$and $\Lambda$
as self-maps on $\MM(2,\C)$ so that for each $A\in\MM(2,\C)$
\begin{equation}\label{eq:svg_decomp}
	A=U(A)\Lambda(A)V^*(A).
\end{equation}
Let $\mathscr D=\{A\in\MM(2,\C): \sigma_1(A)=\sigma_2(A)\}$. Then we have the following simple fact:
\begin{lemma}\label{l:svd_smooth}
	For some suitable choices of column vectors of
	$V(A)$, we have
	$$U,V,\Lambda:\MM(2,\C)\setminus \mathscr D\rightarrow \MM(2,\C)$$ are
	all $C^{\infty}$ maps. Here $C^{\infty}$ is in the sense that all
	these maps are between real manifolds.
\end{lemma}
\begin{proof}
	Let $A=\begin{pmatrix}a&b\\c&d\end{pmatrix}\in \MM(2,\C)\setminus
	\mathscr D$. First, we note that $A^*A=V\Lambda^2 V^*$ which implies
	that $\sigma_1^2(A)$ and $\sigma_2^2(A)$ are two eigenvalues of $A^*A$.
	Hence,
	$$\tr(A^*A)=|a|^2+|b|^2|+|c|^2+|d|^2=\sigma_1^2(A)+\sigma_2^{2}(A)>2|\det(A)|\mbox{ and}$$
	$$\sigma_1^2(A)=\frac{1}{2}\left(\tr(A^*A)+\sqrt{\tr(A^*A)^2-4|\det(A)|^2}\right),$$ 
	which implies that
	$\sigma_1^2(A)$, and hence $\sigma_1(A)$, are $C^{\infty}$ on $\MM(2,\C)\setminus \mathscr D$. This proves that 
	$$
	\Lambda\in C^{\infty}(\MM(2,\C)\setminus \mathscr D, \MM(2,\C)).
	$$ 
	Let $V'$ be a matrix that diagonalizes $A^*A$.
	Then the column vectors of $V'$ are solutions of the equations
	$$(A^*A-\sigma_1^2(A)I_2)\binom{x_1}{x_2}=0\mbox{ and }(A^*A-\sigma_2^{2}(A)I_2)\binom{x_1}{x_2}=0.$$
	In particular, we may choose $V'$ to be 
	$$
	\begin{pmatrix}\bar ab+\bar cd,&\bar ab+\bar
		cd\\\sigma_1^2(A)-|a|^2-|c|^2,&\sigma_2^2(A)-|a|^2-|c|^2\end{pmatrix},
	$$
	which has nonzero determinant since $\sigma_1(A)>\sigma_2(A)$. Then we can
	choose $V$ as $$V(A)=\frac{1}{\sqrt{\det(V'(A))}}V'(A).$$
	This shows that 
	$$
	V\in C^{\infty}(\MM(2,\C)\setminus \mathscr D, \mathrm{SU}(2,\C)).
	$$ 
	Finally, recall that column vectors of $U$ are eigenvectors of $AA^*$ whose eigenvalues are again $\sigma_1^2(A)$ and $\sigma_2^2(A)$. Hence, similar to the process finding $V$,  we may first fix a choice of a unit vector $\vec u_1$ for the first colum vector of $U$.  The second column vector $\vec u_2$ of $U$ is then determined by the facts $\|\vec u_2\|=1$, $\langle \vec u_1, \vec u_2\rangle=0$, and $\det (U)=\frac{\det A}{|\det A|}$. Since all operations involved in determining $U$ are algebraic operations/equations in entries of $A$, $\bar A$, and $\Lambda$, it follows that 
	\[
	U\in C^{\infty}(\MM(2,\C)\setminus \mathscr D, \mathrm{U}(2,\C)).
	\] 
\end{proof}

The most contracted direction $s(A)$ of $A\in\MM(2,\C)\setminus\mathscr D$ is defined as 
\[
s(A)=V(A)\cdot\infty: \MM(2,\C)\setminus\mathscr D\to \C\PP^1;
\]
that is, it is the projection of the second column vectors of $V$. We also define $u(A)$ to be
\[
u(A)=U(A)\cdot 0: \MM(2,\C)\setminus\mathscr D\to \C\PP^1.
\] 
If $\det (A)\neq 0$, then $u(A)$ is precisely $s(A^{-1})$. In any case, we always have $u^\perp(A)=s(A^*)$. By Lemma~\ref{l:svd_smooth}, $s$ and $u$ and $C^\infty$ maps.

\section{Domination Implies $\SVG$ and $\FI$}\label{s:dominationsImplies}

In this section, we assume $B\in\CD\CS$ and try to show $\SVG$ and $\FI$ for $B$, which is relatively easier. Let $\vec{u}(j) \in E^u(j)$ and $\vec{s}(j) \in E^s(j)$ be unit vectors and 
\[D(j)=(\vec{u}(j), \vec{s}(j)) \in \mathrm{M}(2,\C).
\]
 In other words, $\vec{u}(j)$ and $\vec{s}(j)$ are the column vectors of $D(j)$. By condition (c) of Definition~\ref{d:domination}, \eqref{eq:Dist_to_Det}, and \eqref{eq:DistofComplexLine}, we have
\begin{equation}\label{eq:boundedD}
\inf _{j \in \Z}|\operatorname{det} D(j)|=\inf _{j \in \Z} \frac{d\left(E^u(j), E^s(j)\right)}{2}>\frac{\delta}{2}.
\end{equation}
Define $\l_j^+$ and $\l_j^-$ so that
\begin{align*}
B(j) D(j)&=(B(j)\vec{u}(j), B(j)\vec{s}(j))\\
&=(\lambda_j^{+}\vec{u}(j+1),\lambda_j^{-}\vec{s}(j+1))\\
&=(\vec{u}(j+1), \vec{s}(j+1))\begin{pmatrix}\lambda_j^{+} & 0 \\
	0 & \lambda_j^{-}
\end{pmatrix}\\
&= D(j+1)\Lambda(j),
\end{align*}
where we set 
\begin{equation}\label{eq:diagonalized}
\Lambda(j):=\begin{pmatrix}\l_j^+ &0 \\ 0 &\l_j^-\end{pmatrix}.
\end{equation} 
In other words, $B$ is conjugate to $\Lambda$ via $D$:
\begin{equation}\label{eq:B_conj_Lambda}
	D(j+1)^{-1}B(j)D(j)=\Lambda(j).
	\end{equation}
Using \eqref{eq:boundedD} and \eqref{eq:diagonalized}, one readily checks that $\|D\|_\infty<\infty$ and hence $\|\Lambda\|_\infty<\infty$ (abusing the notation slightly, we still let $\|D\|, \|\Lambda\|_\infty<M$) and $\Lambda\in \CD\CS$ where $N$ in condition (b) may be chosen the same as the one for $B$. Moreover, it is not difficult to see that 
\begin{equation}\label{eq:lambdaPlusBoundedBelow}
\eta:=\inf_{j\in\Z}|\l_j^+|>0,
\end{equation}
see e.g. \cite[Remark 1]{alkornzhang} for a proof. Note that we have 
\begin{equation}
B_n(j) D(j)=D(j+n)\Lambda_{n}(j)=(\vec{u}(j+n), \vec{s}(j+n))\begin{pmatrix}
\prod_{k=0}^{n-1} \lambda_{j+k}^{+} & 0 \\
0 & \prod_{k=0}^{n-1} \lambda_{j+k}^{-}
\end{pmatrix}.
\end{equation}

\begin{lemma}\label{l:DominationImpliesWSVG}
Suppose $B$ admits dominated splitting. Then there exist $c>0$ and $\mu>1$ such that  for all $n \in \Z_{+}$ and all $j\in\Z$:
 \begin{equation}
 \|B_{n}(j)\vec u(j)\|> c\mu^n\|B_n(j)\vec s(j)\|,
 \end{equation}
which in particular implies
\begin{equation}\label{eq:exp_gap}
 \sigma_1(B_{n}(j))> c\mu^n\sigma_2(B_{n}(j)).
\end{equation}
\end{lemma}
\begin{proof}

For any $n \in \Z_{+}$,  we let $n=mN+r$ where $0\le r <N$. By (a) of Definition~\ref{d:domination}, it clearly holds for all $k, l\ge 1$ and all $j$ that 
\[
\|B_{k+l}(j)\vec u(j)\|=\|B_k(j+l)\vec u(j+l)\|\cdot \|B_l(j)\vec u(j)\|.
\]
Same holds true for $\vec s(j)$. Hence, by (b) of Definition~\ref{d:domination} it holds that for all $j$ that
\begin{align*}
    \|B_{mN}(j)\vec u(j)\|&=\prod^{m-1}_{k=0}\|B_{N}(j+Nk)\vec u(j+kN)\|\\
       &\ge \lambda^m\prod^{m-1}_{k=0}\|B_{N}(j+Nk)\vec s(j+kN)\|\\
       &\ge \lambda^m\|B_{mN}(j)\vec s(j)\|.
\end{align*}    
If $\|B_{r}(j)\vec s(j)\|=0$, there is nothing to say. So we may assume that $\|B_{r}(j)\vec s(j)\|\neq 0$. It is clear that $\|B_{r}(j)\vec s(j)\|\le \|B_r(j)\|\le M^r$. Combine all these facts, we obtain
\begin{align*}
\|B_n(j)\vec  u(j)\|&=\|B_{r}(j)\vec u(j)\|\cdot\left\|B_{mN}(j+r)\vec u(j+r)\right\|\\
&\ge \l^m\frac{\|B_{r}(j)\vec u(j)\|}{\|B_{r}(j)\vec s(j)\|}\left\|B_{mN}(j+r)\vec s(j+r)\right\|\cdot \|B_r(j)\vec s(j)\|\\
&> \l^m\frac{\|B_{r}(j)\vec u(j)\|}{\|B_{r}(j)\vec s(j)\|}\cdot\|B_{n}(j)\vec s(j)\|\\
&\ge \frac{\prod^{j+r-1}_{k=j}|\l_k^+|}{M^r}\l^{-r/N}(\l^{1/N})^n\|B_n(j)\vec s(j)\|\\
&\ge C\mu^n\|B_n(j)\vec s(j)\|,
\end{align*}
where $C=\min\big\{(\eta M^{-1}\l^{-1/N})^r:0\le r<N-1\big\}$ and $\mu=\l^{1/N}>1$. 

Finally, by \eqref{eq:SingularValues}, we have $\sigma_2(A)\le \|A\vec v\|\le \sigma_1(A)$ for all $A\in\MM(2,\C)$ and all unit vector $\vec v\in\C^2$. Hence,
$$
\sigma_1(B_n(j))\ge \|B_{n}(j)\vec u(j)\|>c\mu^n\|B_{n}(j)\vec s(j)\|\ge c\mu^n\sigma_2(B_n(j)),
$$
as desired.

\end{proof}

Now we begin to prove the only if part of theorem 1. Instead of dealing with $B$, we do it for $\Lambda$ from \eqref{eq:diagonalized} and translate back to $B$ via the conjugation \eqref{eq:B_conj_Lambda}.
\begin{lemma}
$\Lambda$ satisfies $\SVG$.
\end{lemma}
\begin{proof}
Since $\Lambda\in\CD\CS$, by Lemma \ref{l:DominationImpliesWSVG} we have for all $j\in\Z$ and $n\in\Z_+$:
$$
\sigma_1(\Lambda_{n}(j))>c\mu^n\sigma_2(\Lambda_{n}(j)).
$$
On the other hand, condition (b) and \eqref{eq:lambdaPlusBoundedBelow} imply that there is a $\widetilde N$ such that for all $n\ge \widetilde N$, 
\begin{equation}\label{eq:norm=lamba+}
\left|\prod_{k=0}^{n-1} \lambda_{j+k}^{+}\right|=\sigma_1(\Lambda_{n}(j)).
\end{equation}
Thus, for all $n\ge \tilde N$, it holds that 
\[
\sigma_1(\Lambda_{n+1}(j))=\begin{cases}|\lambda_j^+|\sigma_1(\Lambda_n(j+1))\ge c|\eta|\mu^n\sigma_2(\Lambda_n(j+1))\\
|\lambda_{j+n}^+|\sigma_1(\Lambda_n(j))\ge c|\eta|\mu^n\sigma_2(\Lambda_n(j)).
\end{cases}
\]
By choosing $C>0$ appropriately, the estimate above clearly implies:
\[
\sup_{n\ge \widetilde N, j\in\Z} \left\{\frac{\sigma_2(\Lambda_n(j))}{\sigma_1(\Lambda_{n+1}(j))},\ \frac{\sigma_2(\Lambda_n(j+1))}{\sigma_1(\Lambda_{n+1}(j))}\right\}\le C\mu^{-n}.
\]
For $0\le n<\widetilde N$, by \eqref{eq:lambdaPlusBoundedBelow}, we have $\sigma_1(\Lambda_{n+1}(j))\ge \eta^{n+1} $, $\sigma_2(\Lambda_n(j))\le \eta^n$, and choosing $C=\frac1\eta \mu^{\widetilde N} $, we obtain
\[
\sup_{0\le n<\widetilde N, j\in\Z} \left\{\frac{\sigma_2(\Lambda_n(j))}{\sigma_1(\Lambda_{n+1}(j))},\ \frac{\sigma_2(\Lambda_n(j+1))}{\sigma_1(\Lambda_{n+1}(j))}\right\}\le \frac{1}{\eta}\le C\mu^{-n},
\]
which clearly completes the proof.

\end{proof}

\begin{lemma}
$\Lambda$ satisfies \mbox{(FI)}.
\begin{proof}
For $n\ge \widetilde N$, by \eqref{eq:norm=lamba+} we have for all $j\in\Z$
\begin{align*}
&\min\left\{\frac{\|\Lambda_{n+1}(j)\|}{\|\Lambda(j)\|\|\Lambda_n(j+1)\|},\ \frac{\|\Lambda_{n+1}(j)\|}{\|\Lambda_n(j)\|\|\Lambda(j+n)\|}\right\}\\
&=\min\left\{\frac{|\lambda_j^+|}{\|\Lambda(j)\|},\ \frac{|\lambda_{j+n}^+|}{\|\Lambda(n+j)\|}\right\}\\
&>\frac{\eta}{M}.
\end{align*}
For $0\le n<\widetilde N$, we have for all $j\in\Z$:
\[
\min_{0\le n<\widetilde N}\left\{\frac{\|\Lambda_{n+1}(j)\|}{\|\Lambda(j)\|\|\Lambda_n(j+1)\|},\ \frac{\|\Lambda_{n+1}(j)\|}{\|\Lambda_n(j)\|\|\Lambda(j+n)\|}\right\}
	\ge 
	\min_{0\le n<\widetilde N}\left\{\frac{\eta^{n+1}}{M^{n+1}}\right\}.
\]
The two esitmates above actually imply \eqref{eq:sFI} for $\Lambda$ which in particular implies $\FI$, concluding the proof.

\end{proof}
\end{lemma}

\begin{lemma}
 $B$ satisfies $\SVG$ and $\FI$.
\begin{proof}
	One readily sees that if $A=D_1BD_2$ where $\tilde c<\sigma_2(D_i)\le \sigma_1(D_i)<\tilde C,\ i=1,2$, then 
	\[
	c\sigma_i(A)\le \sigma_i(B)\le C\sigma_i(A)
	\]
	where $c, C$ depend only on $\tilde c, \tilde C$. Recall that we have 
	$$
	B_n(j)=D(j+n)\Lambda_{n}(j)D(j)^{-1}.
	$$ 
By the fact that column vectors of $D$ are unit vectors and \eqref{eq:boundedD}, it holds for some $\tilde c>0$ that 
\[
\tilde c<\sigma_2(D^{\pm1}(j))\le \sigma_1(D^{\pm 1}(j))<2\mbox{ for all }j\in\Z.
\]
Hence, we have for $i=1,2$:
\[
c\sigma_i(\Lambda_n(j))\le \sigma_1(B_n(j))<C\sigma_i(\Lambda_n(j))\mbox{ for all }n\ge 0 \mbox{ and }j\in\Z.
\]
This clearly allows one to pass the $\SVG$ and $\FI$ from $\Lambda$ to $B$.
\end{proof}
\end{lemma}

\section{$\SVG$ and $\FI$ Imply Domination}\label{s:obtainDomination}

Let first prove Condition (d) from Defintion~\ref{d:domination} which is relatively straightforward. In fact, we instead prove the stronger version Condition (d').

\begin{lemma}\label{l:uLargeNorm}
	For all $n \in \Z_{+}$, it holds that $\inf_{j\in\Z}\|B_n(j)\|>0$.
	\begin{proof}
		Applying \mbox{(SVG)} with $n=0$, we obtain for all $j\in\Z$:
		$$\frac{\sigma_2(B_0(j))}{\sigma_1(B(j))}=\frac{\sigma_2(I_{2})}{\sigma_1(B(j))}=\frac{1}{\|B(j)\|}<C\mu^{-0}=C,$$ which implies for all $j\in\Z$:
		$$\|B(j)\|>\frac{1}{C}.$$
		Now by \mbox{(FI)}, for each $n\ge 0$ it holds for all $j \in \Z$ that
		 $$ 
		 \frac{\left\|B_{n+1}(j-1)\right\|}{ \left\|B_n(j)\right\|}>c\mu^{-(1-\e)n}.
		 $$
		Set $n=1$, we have $$\|B_{2}(j)\|>c\mu^{\e-1}\|B(j)\|>c\mu^{\e-1}.$$ 
		By induction, we for any $n \geq 0$, it holds for all $j\in\Z$ that
		$$\|B_{n}(j)\|>c\mu^{(\e-1)n(n-1)/2}.$$
		
	\end{proof}
\end{lemma}

\subsection{Existence of $E^s$ and $E^u$}

We need the  following simple but useful lemma.
\begin{lemma}\label{l:z_close_to_s}
	Let $A\in\MM(2,\C)$ be a nonzero matrix and $z\in\C\PP^1$. If $\|A\vec z\|<\delta$, then it holds:
	\[
	d(z, s(A))\le\frac{2(\delta+\sigma_2(A))}{\sigma_1(A)}.
	\]
	\end{lemma}
	\begin{proof}
		Write $\vec z=c_1\vec s(A)+ c_2\vec s^\perp(A)$ where $|c_1|^2+|c_2|^2=1$. By \eqref{eq:Dist_to_Det2}, it clearly holds that
		\begin{align*}
			d(z, s(A)&=2|\det\big(\vec z, \vec s(A)\big)|\\ 
			&=2|\det\big(c_1\vec s(A)+c_2\vec s^{\perp}(A), \vec s(A)\big)|\\
			&=2|c_2|\cdot |\det\big(\vec s^\perp(A), \vec s(A)\big)|\\
			&=2|c_2|.
		\end{align*}
		On the other hand, it holds that
		\begin{align*}
			|c_2\sigma_1(A)|&=\|A c_2\vec s^\perp(A)\|\\
			&=\|A\vec z-c_1A\vec s(A)\|\\
			&\le \|A\vec z\|+\|A\vec s(A)\|\\
			&\le \delta+\sigma_2(A),
			\end{align*}
		which combined with the estimate above clearly imply the desired result.
		\end{proof}
		We also need the following simple fact that it holds for all $A\in\mathrm{GL}(2,\C)$ that
	\begin{equation}\label{eq:sv_A_A-1}
		\sigma_1(A^{-1})=\frac1{\sigma_2(A)}\mbox{ and } \sigma_2(A^{-1})=\frac1{\sigma_1(A)}.
\end{equation}
		
Now let us go back to a bounded sequence $B:\Z\to \MM(2,\C)$ that satisfies both $\SVG$ and $\FI$. For the rest part of this section, our main goal is to show Conditions (a)-(c) for such a $B$ has dominated splitting. First, we define 
		\[
		s_n(j)=s(B_n(j)) \mbox{ and }u_n(j)=u(B_n(n-j)).
		\]
		Note that $s_n(j)$ and $u_n(j)$ are well defined for all $j$ and all $n$ large since $\SVG$ implies that $B_n(j)\notin \mathscr D$ for such $n$ and $j$'s.
\begin{lemma}\label{l:SVGImpliesExistence}
    There exist two $B$-invariant maps $E^u$ and $E^s: \Z \rightarrow \mathbb{C P}^1$ such that
$$ \lim _{n \rightarrow \infty}s_n=E^s\mbox{ and }\lim _{n \rightarrow \infty}u_n=E^u,$$
where the convergence is uniform in $j\in\Z$. 
\end{lemma} 
\begin{proof}

Fix any $j\in \Z$. To find $E^s(j)$, we first note that
\begin{align*}
\|B_{n+1}(j)\vec s_{n}(j)\|=\|B(j+n)B_{n}(j)\vec s_{n}(j)\|&\leq\|B(j+n)\|\sigma_2(B_{n}(j))<M\sigma_2(B_{n}(j)).
\end{align*}
Applying Lemma~\ref{l:z_close_to_s} with $z=s_n(j)$, $A=B_{n+1}(j)$, and $\delta=M\sigma_2(B_{n}(j))$ and by $\SVG$, we obtain
\begin{equation}\label{eq:s_nCauchy}
d\left(s_{n}(j), s_{n+1}(j)\right)\le\frac{2\big(M\sigma_2(B_{n}(j))+\sigma_2(B_{n+1}(j))\big)}{\sigma_1(B_{n+1}(j))}<C\mu^{-n}.
\end{equation}
 Hence, $\left\{s_n(j)\right\}_{n \in \Z}$ is a Cauchy sequence which implies the existence of a $E^s(j)\in\mathbb{C P}^1$ such that
$$
\lim _{n \rightarrow \infty}s_{n}(j)=E^s(j).
$$ 
Moreover, \eqref{eq:s_nCauchy} clearly implies that 
\begin{equation}\label{eq:s_nUconvToEs}
	d(s_n(j), E^s(j))<C\mu^{-n}\mbox{ for  all } j\in\Z.
	\end{equation}
To show $B$-invariance of $E^s$, we consider two cases:
\vskip .2cm

\noindent \textbf{Case s.I}: $\det(B(j))\neq 0$. Set $\vec v=\frac{B(j)^{-1}\vec s_n(j+1)}{\|B(j)^{-1}\vec s_n(j+1)\|}\in B(j)^{-1}\cdot s_n(j+1)$. Noticing that $\|B(j)^{-1}\vec s_n(j+1)\|^{-1}\le 1/\sigma_2(B(j)^{-1})=\sigma_1(B(j))\le M$, we obtain: 
\begin{align*}
	\|B_{n+1}(j)\vec v\|\le M\|B_{n}(j+1)\vec s_{n}(j+1)\|= M\sigma_2(B_{n}(j+1)).
\end{align*}
Applying Lemma~\ref{l:z_close_to_s} with $z=B(j)^{-1}\cdot s_n(j+1)$, $A=B_{n+1}(j)$, and $\delta=M\sigma_2(B_{n}(j+1))$ and by $\SVG$, we obtain
$$
d(B(j)^{-1}\cdot s_n(j+1),s_{n+1}(j))\le \frac{2M\sigma_2(B_n(j+1))+2\sigma_2(B_{n+1}(j))}{\sigma_1(B_{n+1}(j))}\le C\mu^{-n},
$$
which implies
\begin{align*}
B(j)^{-1}\cdot E^s(j+1)&=B(j)^{-1}\cdot\lim _{n\to\infty} s_n(j+1)\\
&=\lim _{n\to \infty} B(j)^{-1}\cdot s_n(j+1)\\
&=\lim _{n\to \infty} s_{n+1}(j)\\
&=E^s(j) .
\end{align*}
\vskip .2cm

\noindent \textbf{Case s.II}: $\det(B(j))= 0$. Then it is clear that $s_n(j)=\ker (B(j))$ for all $n\ge 1$ which implies $E^s(j)=\ker (B(j))$ and
$$
B(j)\left[E^s\left(j\right)\right]=\{\vec 0\} \subset E^s\left(j+1\right)
$$

Now we consider $E^u(j)$. Recall that $u_n(j)^\perp$ is the most contracted direction of  $B^*_n(n-j)$, which implies that
\begin{align*}
\|B_{n+1}^*(j-n-1)\vec u_{n}^\perp(j)\|&=\|B^*(j-n-1)B^*_{n}(j-n)\vec u_{n}^\perp(j)\|\le M\sigma_2(B_n(j-n)).
\end{align*}
Applying Lemma~\ref{l:z_close_to_s} with $z=u^\perp_n(j)$, $A=B^*_{n+1}(j-n-1)$, and $\delta=M\sigma_2(B_{n}(j-n))$ and by $\SVG$, we obtain that
\begin{align*}
d(u^\perp_{n}(j), u^\perp_{n+1}(j))<\frac{2M\sigma_2(B_{n}(j-n))+2\sigma_2(B_{n+1}(j-n-1))}{\sigma_1(B_{n+1}(j-n-1))}<C\mu^{-n}.
\end{align*}
By \eqref{eq:orth_prev_dist}, we then get
$$
d(u_{n}(j), u_{n+1}(j))=d(u^\perp_{n}(j), u^\perp_{n+1}(j))<C\mu^{-n},
$$
which implies the existence of $E^u: \mathbb{Z} \to \mathbb{C P}^1$ such that
$\lim\limits_{n \rightarrow \infty}u_{n}(j)=E^u(j)$ and
\begin{equation}\label{eq:u_nUconvToEu}
d(u_n(j), E^u(j))<C\mu^{-n}.	
	\end{equation}
Note that by the fact $\mathscr O(z)=z^\perp$ is an isometry (hence, continuous) on $\C\PP^1$, we also have 
\begin{equation}\label{eq:u_perpConvToEu_perp}
	\lim_{n\to\infty}u^\perp_n(j)=[E^u(j)]^\perp
	\end{equation}
To show that $E^u$ is $B$-invarint, we again consider two case:
\vskip .2cm

\noindent \textbf{Case u.I}: $\det(B(j))\neq 0$. Set $\vec v= \frac{[B^*(j)]^{-1}\vec  u^\perp_{n}(j)}{\|[B^*(j)]^{-1}\vec  u^\perp_{n}(j)\|}\in [B^*(j)]^{-1}\cdot u^\perp_{n}(j)$. It clearly holds that
\begin{align*}
\|B^*_{(n+1)}(j-n)\vec v\|&=\frac{\|B^*_{n}(j-n)B^*(j)[B^*(j)]^{-1}\vec  u^\perp_{n}(j)\|}{\|[B^*(j)]^{-1}\vec  u^\perp_{n}(j)\|}\\
&\le M\|B^*_{n}(j-n)\vec u^\perp_n(j)\|\\
&=M\sigma_2(B_n(j-n)).
\end{align*}
Applying Lemma~\ref{l:z_close_to_s} with $z=[B^*(j)]^{-1}\cdot u^\perp_{n}(j)$, $A=B^*_{(n+1)}(j-n)$, and $\delta=M\sigma_2(B_n(j-n))$ and by $\SVG$, we obtain that
\[
	d(B^*(j)]^{-1}\cdot u^\perp_{n}(j), u^\perp_{n+1}(j+1))\le \frac{2\big(M\sigma_2(B_{n}(j-n))+\sigma_2(B_{n+1}(j-n))\big)}{\sigma_1(B_{n+1}(j-n))}<C\mu^{-n}.
\]
Letting $n$ tends to $\infty$, we obtain $B^*(j)\cdot [E^u(j+1)]^\perp=[E^u(j)]^\perp$, which in turn implies
\[
	\langle \vec E^\perp_{n}(j+1), B(j)\vec E^u(j)\rangle=\langle B^*(j)\vec E^\perp_{n}(j+1), \vec E^u(j)\rangle=0,
\]
and hence, $B(j)\cdot  E^u(j)=E^u(j+1)$.
\vskip .2cm

\noindent \textbf{Case u.II}: $\det(B(j))= 0$. Then it is clear that the most contracted direction $u^\perp_n(j+1)$ of $B^*_n(j+1-n)=B^*_{n-1}(j+1-n)B^*(j)$ is $\ker [B^*(j)]$ for all $n\ge 1$ which implies 
\[
[E^u(j+1)]^\perp =\ker (B^*(j)),
\]
which in turn implies that for all $\vec v\in\C^2$:
\[
\langle \vec E^\perp_{n}(j+1), B(j)\vec v\rangle=\langle B^*(j)\vec E^\perp_{n}(j+1),v\rangle=0,
\]
and hence $E^u(j+1)=\mathrm{Im}(B(j))=B(j)(\C^2)$. In particular, $B(j)[E^u(j)]\subset E^u(j+1)$.
\end{proof}

\begin{remark}\label{r:InvDirectionSingularCase}
	Although we may not need the following facts, it is worthwhile to point them out. By the proof of Lemma~\ref{l:SVGImpliesExistence}, we notice that if $\det(B(j))=0$, then $E^s(j)=\ker(B(j))$ and $E^u(j+1)=\mathrm{Im}(B(j))$. In fact, the same proof provides us with more information. Using the same $j$ as above and we have for all $j'\le j$:
	\begin{equation}\label{eq:Es_singularCase}
	E^s(j')=\ker(B_{j-j'+1}(j')).
	\end{equation}
	Indeed, we just need to notice that in this case, $s_n(j')=\ker(B_{j-j'+1}(j'))$ for all $n\ge j-j'+1$. Likewise, we have for all $j'>j$:
		\begin{equation}\label{eq:Eu_singularCase}
		E^u(j')=\mathrm{Im}(B_{j'-j}(j)).
	\end{equation}
	Indeed, in this case, like in the proof of Lemma~\ref{l:SVGImpliesExistence}, we have for all $n> j'-j$:
	\[u^\perp_n(j')=s(B^*_n(j'-n))=s(B^*_n(n-j'+j)B^*_{j'-j}(j))=\ker (B^*_{j'-j}(j)),
	\] 
which implies that $u_n(j')=\mathrm{Im}(B_{j'-j}(j))$ for all such $n>j'-j$.
	\end{remark}
	
\subsection{Separation of $E^s$ and $E^u$}

Now we consider to show that $\inf_{j\in\Z}d(E^u(j), E^s(j))>0.$ We divide the proof into two steps: first, we show $E^u(j)\neq E^s(j)$ for all $j$ in case $\det(B(j))=0$ for some; second, we show that $E^u(j)\neq E^s(j)$ for all other cases. It turns out the argument in the second step naturally gives us that $d(E^u(j), E^s(j))$ is bounded away from $0$ uniformly. Once we have $\inf_{j\in\Z}d(E^s(j), E^u(j))>0$, Condition (b) follows easily from $\SVG$.

\begin{lemma}\label{l:EsNeqEu_singularCase}
    Suppose $\det(B(j))=0$ for some $j$. Then $E^u(j)\neq E^s(j)$ for all $j\in\Z$.
\begin{proof}
Fix any arbitrarily $j$. We have two different cases: Case I, there is a $j'<j$ so that $\det(B(j'))=0$; or Case II, there is a $j'\ge j$ so that $\det(B(j'))=0$. By $B$-invariance of $E^{u(s)}$ and the fact that $B(j')E^{u(s)}(j')=E^{u(s)}(j')$ when $\det(B(j'))\neq 0$, we may further reduce the two cases to: Case I: $\det (B(j-1))=0$; Case II: $\det(B(j))=0$. For instance in Case I, if we set $j_0=\max\{j'<j: \det(B(j))=0\}<j-1$, then we have by the facts above $E^s(j)=E^u(j)$ if and only if 
\[
E^s(j_0)=B_{j_0-j}(j)E^s(j)=B_{j_0-j}(j)E^u(j)=E^u(j_0).
\]
For simplicity, from now on we write unit vectos in $E^s(j)$ and $E^u(j)$ as $\vec s(j)$ and $\vec u(j)$, respectively. Likewise, $\vec s^\perp(j)$and $\vec u^\perp (j)$ are unit vectors that are orthogonal to $\vec s(j)$ and $\vec u(j)$, respectively.
\vskip .2cm
   
\noindent \textbf{Case I}. Since $\det (B(j-1))=0$, we have  $E^u(j)=\mathrm{Im}(B(j-1))$ by Remark~\ref{r:InvDirectionSingularCase}. Hence, $\frac{B(j-1)\vec v}{\|B(j-1)\vec v\|}$ can be $\vec u(j)$ for all $\vec v\notin \ker (B(j-1))$. In particular, we can choose $\vec v=\vec s^{\perp}_{n+1}(j-1)$ which belongs to the most expanding direction of $B_{n+1}(j-1)$ (hence cannot be in $\ker(B(j-1))$).
then we obtain 
$$
\|B_{n}(j)\vec u(j)\|=\frac{\|B_{n+1}(j-1)\vec s^\perp_{n+1}(j-1)\|}{\|B(j-1)\vec v\|}\ge \frac{\sigma_1(B_{n+1}(j-1))}{M}.
$$

Now we let $\vec s(j)=c_1\vec s_{n}(j)+c_2\vec s^{\perp}_{n}(j)$, where $|c_1|^2+|c_2|^2=1$ and $|c_2|=d(E^s(j), s_n(j))<C\mu^{-n}$. Then we have
\begin{align*}
\|B_{n}(j)\vec s(j)\|&\le |c_1|\|B_{n}(j)\vec s_{n}(j)\|+|c_2|\|B_{n}(j)\vec s^{\perp}_{n}(j)\|\\
&\le \sigma_2(B_n(j))+C\mu^{-n}\sigma_1(B_n(j)).
\end{align*}
Combine the two estimates above together with $\SVG$ and $\FI$, we obtain
\begin{align*}
\frac{\|B_n(j)\vec s(j)\|}{\|B_n(j)\vec u(j)\|}&\le M\frac{\sigma_2(B_n(j))+C\mu^{-n}\sigma_1(B_n(j))}{\sigma_1(B_{n+1}(j-1))}\\
&\le M\left[\frac{\sigma_2(B_n(j))}{\sigma_1(B_{n+1}(j-1))}+C\mu^{-n}\frac{\sigma_1(B_n(j))}{\sigma_1(B_{n+1}(j-1))}\right]\\
&\le CM\mu^{-n}+C\mu^{-n}\mu^{(1-\e)n}\\
&\le C\mu^{-\e n}<1,
\end{align*}
for $n$ sufficiently large, which clearly implies that $E^u(j)\neq E^s(j)$.
\vskip .2cm
\noindent \textbf{Case II}. To show $E^s(j)\neq E^u(j)$, we equivalently show $[E^s(j)]^\perp\neq [E^u(j)]^\perp$. Since $\det(B(j))=0$, we have $E^s(j)=\ker(B(j))$ by Remark~\ref{r:InvDirectionSingularCase}. This is clearly equivalent to $[E^s(j)]^\perp=\mathrm{Im}(B^*(j))$. Hence, $\frac{B^*(j)\vec v}{\|B^*(j)\vec v\|}$ can be our $\vec s^\perp(j)$ for all $\vec v\notin\ker(B^*(j))$. In particular, we may choose $\vec v=u_n(j+1)$ which is the most expanding direction of $B^*_{n+1}(j-n)=B^*_n(j-n)B^*(j)$ (hence cannot be in $\ker(B^*(j))$). This  implies that
\[
\|B^*_n(j-n)\vec s^\perp(j)\|=\frac{\|B^*_{n+1}(j-n)\vec u_n(j+1)\|}{\|B^*(j)\vec v\|}\ge\frac{\sigma_1(B_{n+1}(j-n))}{M}.
\] 

On the other hand, recall that we have  $d([E^u(j)]^\perp, u^\perp_n(j))<C\mu^{-n}$ where $u^\perp_n(j)$ is the most contracted direction of $B^*_{n}(j-n)$. Now we write $\vec u^\perp(j)=c_1\vec u^\perp_n(j)+c_2\vec u_n(j)$ where $|c_1|^2+|c_2|^2=1$ and $|c_2|=d([E^u(j)]^\perp, u^\perp_n(j))$. Then we have
\begin{align*}
\|B^*_{n}(j-n)\vec u^\perp(j)\|&\le |c_1|\|B_{n}(j-n)\vec u^\perp_{n}(j)\|+|c_2|\|B_{n}(j-n)\vec u_{n}(j)\|\\
&\le \sigma_2(B_n(j-n))+C\mu^{-n}\sigma_1(B_n(j-n)).
\end{align*}

The same proof as in Case I shows that it holds for all large $n$:
\[
\frac{\|B^*_{n}(j-n)\vec u^\perp(j)\|}{\|B^*_n(j-n)\vec s^\perp(j)\|}<1,
\]
which implies that $[E^s(j)]^\perp\neq [E^u(j)]^\perp$, hence $E^s(j)\neq E^u(j)$.

\end{proof}
\end{lemma}

Now we introduce the definition of $\MM(2, \C)$-cocycles. Let $\Omega$ be a compact metric space, $T$ be a homeomorphism in $\Omega$, and $A:\Omega\to \MM(2, \C)$ be continuous . Then we consider the following dynamical system:
$$
(T, A): \Omega \times \mathbb{C}^2 \rightarrow \Omega \times \mathbb{C}^2,(T, A)(\omega, \vec{v})=(T \omega, A(\omega) \vec{v}) .
$$

Iterations of dynamics are denoted by $\left(T^n, A_n\right):=(T, A)^n$. In particular, we have
$$
A_n(\omega)= \begin{cases}B\left(T^{n-1} \omega\right) \cdots B(\omega), & n \geq 1, \\ I_2, & n=0,\end{cases}
$$ and $A_{-n}(\omega)=\left[A_n\left(T^{-n} \omega\right)\right]^{-1}, n \geq 1$, where all matrices involved are invertible.  Similar to $\MM(2,\C)$-sequences, we can define $\SVG$ and $\FI$ for $\MM(2,\C)$-cocycles as follows:
\begin{align*}
	&\SVG\hskip 1.65cm \sup_{\omega\in\Omega}\left\{\frac{\sigma_2(A_n(\omega))}{\sigma_1(A_{n+1}(\omega))},\  \frac{\sigma_2(A_n(T\omega))}{\sigma_1(A_{n+1}(\omega))}\right\}<C\mu^{-n} \mbox{ for all } n\ge 0;\\
		&\FI\hskip 2cm \sup_{\omega\in\Omega}\left\{\frac{\sigma_1(A_n(\omega))}{\sigma_1(A_{n+1}(\omega))},\  \frac{\sigma_1(A_n(T\omega))}{\sigma_1(A_{n+1}(\omega))}\right\}<C\mu^{(1-\e)n} \mbox{ for all } n\ge 1,
\end{align*}
where we use the same parameters and notations as in the case of $\MM(2,\C)$-sequences.

We can actually go from a bounded $\MM(2\C)$ sequence to $\MM(2,\C)$-cocycle as follows. Let $\mathcal{B}^{\mathbb{Z}}$ be the space of full shift generated by a set of alphbets $\mathcal{B}$. Suppose $\mathcal{B}$ is a compact topological space and $\mathcal{B}^{\mathbb{Z}}$ be equipped with the product topology, then $\mathcal{B}^{\mathbb{Z}}$ is a compact topologic space as well. Let $T: \mathcal{B}^{\Z}  \rightarrow \mathcal{B}^{\Z} $ be the operator of left shift, i.e.

$$
(T \omega)_{n}=\omega_{n+1} \text { for } \omega=(\omega_{n})_{n \in \Z }\in \mathcal{B}^{\mathbb{Z}}
$$

\begin{defi}
For each $\omega \in \mathcal{B}^{\mathbb{Z}}$, the Hull of $\omega$ is defined  as $\overline{\{T^n(\omega):\ n \in \Z\}}$, i.e. the closure of the $T$-orbit of $\omega$ under the product topology, and is denoted by $\mathrm{Hull}(\omega)$. Clearly, $\mathrm{Hull}(\omega)$ is a compact topological space that is invariant under $T$.
\end{defi}

Now take $\mathcal{B}=B_M[\mathrm{M}(2, \mathbb{C})]$ where $B_M$ denotes the ball in $\mathrm{M}(2, \mathbb{C})$ with operator norm less than or equal to $M$. Then pick an element $B: \mathbb{Z} \rightarrow B_M[\mathrm{M}(2, \mathbb{C})]$ in $\mathcal{B}^{\mathbb{Z}}$ and  set $\Omega=\overline{\left\{T^n(B)\right\}_{n \in \mathbb{Z}}}=\operatorname{Hull}(B)$. Clearly, $T: \Omega \rightarrow \Omega$ is a homeomorphism. Let $F: \Omega \rightarrow \operatorname{M}(2, \mathbb{C})$ be the evaluation map at the 0-position, i.e. $F(\omega)=\omega_{0}$. Then consider the cocycle $(T, F): \Omega \times \mathbb{C}^2 \rightarrow \Omega \times \mathbb{C}^2$ as $\left(T^n, F_n\right)=(T, F)^n$ with $F_n(\omega)=\omega_{n-1} \cdots \omega_0$. 

\begin{prop}\label{p:seq_to_cocycle_svgFI}
Let $B$ and $(\Omega, T, F)$ be as above. Then $B$ satisfies $\SVG$ and $\FI$ if and only if so is $F$.

\begin{proof}
The if part is obvious via the relation $B_n(j)=F_n(T^j B)$. For the only if part, we first note that in product topology $\omega^{(k)}$ converges to $\omega$ means pointwise convergence. In other words, $\omega^{(k)}_n$ converges to $\omega_n$ as $k\to\infty$ for each $n\in\Z$. 

Now for every $\omega\in\Omega$, we can find $j$ such that $T^j(B)$ is sufficiently close to $\omega$ since $T^j(B)$ is dense in $\Omega=\mathrm{Hull}(B)$. In particular, for every $n\ge 0$, we can choose $j$ so that $B_n(j)=F_n(T^jB)$ to be sufficiently close to $F_n(\omega)$. Now by the fact $\MM(2,\C)\setminus\mathscr D$ is an open set where $\sigma_1$ and $\sigma_2$ are continuous, we can then pass both $\SVG$ and $\FI$ from $B$ to $F$. Indeed, taking $\SVG$ as an example, we have for all $n\ge 0$ and all $j\in\Z$:
\begin{equation}\label{eq:svg_nj}
\frac{\sigma_2(B_n(j))}{\sigma_1(B_{n+1}(j))}<C\mu^{-n}.
\end{equation}
By choosing $j$ so that $B_n(j)$ is sufficiently close to $F_n(\omega)$, we first obtain $F_n(\omega)\notin\mathscr D$ since so is $B_n(j)$. This in turn implies that $\sigma_i(B_n(j))$ is sufficiently close to $\sigma_i(B_n(j))$ for $i=1,2$. Hence, \eqref{eq:svg_nj} implies that
\[
\frac{\sigma_2(F_n(\omega))}{\sigma_1(F_{n+1}(\omega))}<C\mu^{-n}.
\]
Note that the estimate above is independent of $\omega\in\Omega$. All other inequalities contained in $F$'s $\SVG$ and $\FI$ follows from the same fashion as they all only involve $\sigma_1(F_n(\omega))$ and $\sigma_2(F_n(\omega))$ for some $n\ge 0$.

\end{proof}
\end{prop}

\begin{lemma}\label{l:EsAwayFromEu}
Suppose $\det(B(j))\neq 0$ for all $j\in\Z$. Then $E^s(j)\neq E^u(j)$ for all $j\in\Z$. Moreover, for all bounded $B:\Z\to\MM(2,\C)$ satisfying $\SVG$ and $\FI$, we have $\inf_{j\in\Z}d(E^u(j), E^s(j))>0$.
\end{lemma}
\begin{proof}
To prove the first part, we split it into two different cases. 

\noindent \textbf{Case I}: $\inf_{j\in\Z}|\det(B(j))|>0$. As we mentioned after Definition~\ref{d:domination}, this case can be reduced to $\mathrm{SL}(2,\C)$ which allows us to use the proof of \cite[Lemma 1]{zhang2}. We include a proof for the sake of completeness. Define a new sequence 
\[
A(j)=\frac1{\sqrt{\det(B(j))}}B(j)\in\mathrm{SL}(2,\C).
\]
It is clear that $A$ and $B$ share the same $E^s$ and $E^u$. It is also clear that $\|A\|_\infty<\infty$. We may still assume that $\|A(j)\|<M$ for all $j\in\Z$. Moreover, it is clear that  $\SVG$ is equivalent to the following uniform exponential growth condition in this case: there is a $\l>1$ such that
\begin{equation}\label{eq:ueg}
\mathrm{(UEG)}_{c,\l} \hskip .5cm \inf_{j\in\Z}\|A_n(j)\|\ge c\lambda^n\mbox{ for all }n\ge 1.
\end{equation}
In fact, we may simply choose $\l=\sqrt{\mu}$. Now we claim that there exists a pair $(c, \lambda)$ such that $A$ satisfies both $\mathrm{(UEG)}_{c, \lambda}$ and:
\begin{equation}\label{eq:upbound}
\mbox{for all }N\ge 1, \mbox{ there exists a } j_0\in\Z\mbox{ and }n_0\ge N\mbox{ such that }\|A_{n_0}(j_0)\|<c\lambda^{\frac32n_0}.
\end{equation}
Indeed, let's begin with a pair $(c_0,\l_0)$ so that we have $\mathrm{(UEG)}_{c_0,\l_0}$. If \eqref{eq:upbound} holds true for $(c_0, \l_0)$, then we are done. Otherwise, there is a $N_1$ such that 
\[
\|A_n(j)\|\ge c_0\l_0^{3/2} \mbox{ for all } j\ge N_1.
\]
Now we set $\lambda_1=\l_0^{\frac32}$ and $c_1=\min\{c_0, \lambda_1^{-N_1}\}$. We claim that $A$ satisfies $\mathrm{(UEG)}_{c_1,\lambda_1}$. Indeed, for $n<N_1$, it is a trivial estimate that
\[
\|A_n(j)\|\ge 1\ge c_1\l^{N_1}\ge c_1\l^{n}.
\]
For $n\ge N_1$, it clearly holds that 
\[
\|A_n(j)\|\ge c_0\lambda^{\frac32 n}\ge c_1\l_1^{\frac32}.
\] 
We may repeat this process with the new pair $(c_1, \l_1)$. This process must terminate at step $k$ where $\l_0^{(3/2)^k}>M$. Thus we obtain a pair $(c,\l)$ where we have both $\mathrm{(UEG)_{c,\l}}$ and \eqref{eq:upbound}. From now on, we work with such a pair. Let $j_0$ and $n_0\ge N$ be from \eqref{eq:upbound} for some $N\ge 1$. Note that $\mathrm{(UEG)}_{c,\lambda}$ is equivalent to $\SVG$ with $\mu=\lambda^2$. Hence, by \eqref{eq:s_nUconvToEs}, we have 
\[
d(E^s(j_0), s_n(j_0))<C\lambda^{-2n}.
\]
Recall that $\vec s(j)$ denotes a unit vector in $E^s(j)$. Write
\[
\vec s(j_0)=c_1\vec s_n(j_0)+c_2\vec s^\perp_n(j_0),
\]
where $|c_1|^2+|c_2|^2=1$ and $|c_2|=d(E^s(j_0), s_n(j_0))$. Then we have 
\begin{align*}
\|A_{n_0}(j_0)\vec s(j_0)\|&\le \|A_n(j_0)\vec s_n(j_0)\|+|c_2|\|A_n(j_0)\vec s^\perp_n(j_0)\|\\
&\le \|A_{n_0}(j_0)\|^{-1}+C\lambda^{-2n_0}\|A_{n_0}(j_0)\|\\
&\le C\lambda^{-n_0}+C\l^{3/2n_0}\l^{-2n_0}\\
&\le C\lambda^{-n_0/2}<1,
\end{align*}
where in the last step we simply choose $N$ large enough. Replacing $A_n(j_0)$, $E^s(j_0)$, and $s_n(j_0)$ by $A_{-n_0}(j_0+n_0)$, $E^u(j_0+n_0)$, and $u_{n_0}(j+n_0)$ respectively, the same process above yields
\[
\|A_{-n_0}(j_0+n_0)\vec u(j+n_0)\|\le C\lambda^{-n_0/2}<1.
\]
Since $A_{n_0}(j_0)[E^u(j_0)]=E^u(j_0+n_0)$ and $A_{-n_0}(j_0+n_0)=[A_{n_0}(j_0)]^{-1}$, we obtain
\[
\|A_{n_0}(j)\vec u(j_0)\|=\frac{\|A_{n_0}(j_0)A_{-n_0}(j_0+n_0)\vec u(j+n_0)\|}{\|A_{-n_0}(j_0+n_0)\vec u_{n_0}(j+n_0)\|}=\frac1{\|A_{-n_0}(j_0+n_0)\vec u_{n_0}(j+n_0)\|}>1,
\]
which clearly implies $E^s(j_0)\neq E^u(j_0)$. By $A$-invariance, we then extend it to $j\in\Z$.
\vskip .4cm

 \noindent \textbf{Case II }: $\inf_{j\in\Z}|\det(B(j))|=0$. Let $(\Omega, T, F)$ be as in Prop~\ref{p:seq_to_cocycle_svgFI}. Then we must have some $\tilde\omega$ so that $\det(\tilde \omega_0)=0$. Indeed, if this is not true, then $|\det(\omega_0)|>0$ for all $\omega\in\Omega$. It is clear that $|\det(\omega_0)|$ is continuous in $\omega$. This implies that $\inf_{\omega}|\det(\omega_0)|>0$ which in particular implies $\inf_{j\in\Z}|\det(B(j))|>0$ since $B(j)=T^j(B)(0)$ and $T^jB\in\Omega$.

For each $\omega \in \Omega$, define $B^{\omega}:\Z\to\MM(2,\C)$ so that $B^\omega(j)=F(T^j \omega)$ for all $j \in \Z$. By Proposition~\ref{p:seq_to_cocycle_svgFI}, $B^\omega$ satisfies $\SVG$ and $\FI$ with the same constant $\mu$ for all $\omega$. In paticular, we may let $s_n^\omega(j)=s(B^{\omega}_n(j))$ and $u^\omega_n(j)=s([B_{n}^{\omega}(j-n)]^*)$. Then we define $s_n(\omega)=s(F_n(\omega))$ and $u_n(\omega)=s(F_{-n}(\omega))$. It is clear that we have $s_n(\omega)=s_n^\omega(0)$ and $u_n(\omega)=u_n^\omega(0).$

Now for each $\omega$, we could treat the sequence $\left\{B^\omega(j), j \in \Z\right\}$ as the sequence $\{B(j), j \in \Z\}$. By Lemma~\ref{l:SVGImpliesExistence} and taking $j=0$ for each $\omega$, we obtain
$$
d(s_n(\omega),s_{n+1}(\omega))<C \mu^{-n} \text { and } d(u_n(\omega), u_{n+1}(\omega))<C \mu^{-n}.
$$
Hence, there exist $E^s, E^u: \Omega \rightarrow \C\PP^1$ such that
$$
\lim _{n \rightarrow \infty} s_n=E^s \text { and } \lim _{n \rightarrow \infty} u_n=E^u
$$
and it holds for all $n\ge 1$ that
\begin{equation}\label{eq:Cocycle_unUnivConvEu}
\sup_{\omega\in\Omega}d(s_n(\omega),E^s(\omega)<C \mu^{-n}\mbox{ and } \sup_{\omega\in\Omega}d(u_n(\omega),E^u(\omega)<C \mu^{-n}.
\end{equation}
In particular, for each $\omega$, it holds that
$$
F(\omega) \cdot E^s(\omega)=B^\omega(0) (E^s)^\omega(0)=(E^s)^\omega(1)=E^s(T \omega);
$$
that is, $E^s$ is $(T, F)$-invariant. Similarly, $E^u$ is $(T, F)$-invariant as well. Moreover, $F_n(\omega)\notin\mathscr D$ for all $\omega$ and all large $n$ which together with Lemma~\ref{l:svd_smooth} implies that $s_n(\omega)$ and $u_n(\omega)$ are continuous in $\omega$. By \eqref{eq:Cocycle_unUnivConvEu}, $E^s$ and $E^u$ are continuous maps. Now assume for the sake of contradiciton that $E^s(j)=E^u(j)$ for some $j$. By the $B$-invariance and the fact $\det(B(j))\neq 0$ for all $j$, we obtain that $E^s(j)=E^s(j)$ for all $j$. This in particular means $E^s(T^jB)=E^s(T^jB)$. By continuity of $E^s(\omega)$ and $E^u(\omega)$ and the fact that $\{T^jB: j\in\Z\}$ is dense in $\Omega$, we obtain that $E^s(\omega)=E^u(\omega)$ for all $\omega\in\Omega$.

On the other hand, applying Lemma~\ref{l:EsNeqEu_singularCase} to $\tilde\omega$, we obtain $E^s(\tilde\omega)\neq E^u(\tilde\omega)$. This clearly contradicts $E^s(j)=E^u(j)$ for some $j$ which concludes the proof of the first part of this lemma.

Now apply Lemma~\ref{l:EsNeqEu_singularCase} and the part we have just proved to all $B^\omega(j)=F(T^j\omega)$, we obtain $E^s(\omega)\neq E^u(\omega)$ for all $\omega$. By continuity of $E^s$ and $E^u$ and compactness of $\Omega$, we obtain
$$
\inf d(E^u(\omega), E^s(\omega)>0.
$$
Since $E^{u(s)}(j)=E^{u(s)}(T^jB)$, we then obtain
$$
\inf_{j\in\Z}d(E^u(j), E^s(j))>0,
$$
as desired.
\end{proof}

Finally we show that $E^u$ dominates $E^s$ as stated in condition (b) of Definition~\ref{d:domination}.
\begin{lemma}\label{l:EuDominatesEs}
	There exists a $\lambda>1$ and $N\in Z_+$ such that it holds for all $j\in\Z$ and all unit vectors $\vec s(j)\in E^s(j)$ and $\vec u(j)\in E^u(j)$ that
	\[
	\|B_N(j)\vec u(j)\|>\lambda \|B_N(j) \vec s(j)\|.
	\]
	\end{lemma}
\begin{proof}
	By Lemma~\ref{l:EsAwayFromEu}, there is a $\delta>0$ such that $d(E^u(j), E^s(j))>\delta$ for all $j\in\Z$. This together with \eqref{eq:s_nUconvToEs} and \eqref{eq:u_nUconvToEu} implies that for all $n$ large:
	\[
	d(s_n(j), E^u(j))>\frac\delta2.
	\]
	Writing $\vec u(j)=c_1\vec s_n(j)+c_2\vec s^\perp_n(j)$ and $\vec s(j)=d_1\vec s_n(j)+d_2\vec s^\perp_n(j)$ where $|d_2|=d(E^s(j), s_n(j))$ and $|c_2|=d(E^u(j), s_n(j))$ , we obtain for all $n$ large that
	\begin{align*}
	\frac{\|B_n(j)\vec s(j)\|}{\|B_n(j)\vec u(j)\|}&\le \frac{|d_1|\sigma_2(B_n(j))+|d_2|\sigma_1(B_n(j))}{|c_2|\sigma_1(B_n(j))-|c_1|\sigma_2(B_n(j))}\\
	&= \frac{|d_1|\frac{\sigma_2(B_n(j))}{\sigma_1(B_n(j))}+|d_2|}{|c_2|-|c_1|\frac{\sigma_2(B_n(j))}{\sigma_1(B_n(j))}}\\
	&\le \frac{C\mu^{-n}}{\delta/2-C\mu^{-n}}.
	\end{align*}
The desired result follows by choosing $N$ large so that $C\mu^{-n}<\frac\delta4$ and $\lambda=2$. 
	\end{proof}
	
	\subsection{Domination for $\MM(2,\C)$-Cocycles}
It is worthwhile to point out the equaivalence between domination of dynamically $\MM(2,\C)$-cocycles and their $\SVG$ and $\FI$. Let $(\Omega, T)$ and $A\in C(\Omega, \MM(2,\C))$ be the same as we described right after Lemma~\ref{l:EsNeqEu_singularCase}. Recall the following definition of dominated splitting for $\MM(2,\C)$-cocycles from \cite[Section 5]{alkornzhang}:
	
	\begin{defi}\label{d:domination_dynamical}
		Let $(\Omega,T)$ and $B$ be as above. Then we say $(T,B)$ has dominated spliting if there are two maps $E^s, E^u:\Omega\to \C\PP^1$ with the following properties:
		\begin{enumerate}[(a)]
			\item $E^s, E^u\in C(\Omega, \C\PP^1)$. In other words, they are continuous.
			\item $B(\omega)[E^s(\omega)]\subseteq E^s(T\omega)$ and $B(\omega)[E^u(\omega)]\subseteq E^u(T\omega)$ for all $\omega\in\Omega$. 
			\item There is a $N\in\Z_+$ and $\l>1$ such that  
			$$
			\|B_N(\omega)\vec u\|> \l \|B_N(\omega)\vec s\|
			$$
			for all $\omega\in\Omega$ and all unit vectors $\vec u\in E^u(\omega)$ and $\vec s\in E^s(\omega)$.
		\end{enumerate}
	\end{defi}
As noted in \cite[Lemma 13]{alkornzhang}, we have the following facts. First, we have $B_\omega(j)=A(T^j\omega)$ satsifies Definition~\ref{d:domination} for all $\omega\in\Omega$. Moreover, it holds that:
\begin{prop}\label{p:equiv_conditions_DS_dyna}
	In the context of Definition~\ref{d:domination_dynamical}, condition (a) in Definition~\ref{d:domination_dynamical} is equivalent to 
\[
\inf_{\omega\in\Omega}d(E^s(\omega),E^u(\omega))>0.
\]
\end{prop}

\begin{theorem}\label{t:domination=svg_fi:cocycles}
	$(T,A)$ has dominated splitting if and only if it satisfies both $\SVG$ and $\FI$.
	\end{theorem}
	\begin{proof}
		The only if part follows from the same argument of Section~\ref{s:dominationsImplies}. Indeed, we only need to replace $j\in\Z$ by $\omega\in\Omega$.
		
		For the if part, we bascially apply the argument of Section~\ref{s:obtainDomination} to each sequence $B_\omega(j)=A(T^j\omega)$. Moreover concretely, just like in the proof of Lemma~\ref{l:EsAwayFromEu}, we obtain $E^s(\omega)$ and $E^u(\omega)$ as the limits of $s_n(\omega)=s(A_n(\omega))$ and $u_n(\omega)=s(A^*_n(T^{-n}\omega))^\perp$, respectively, where the convergence in uniform in $\omega\in\Omega$. The same proof implies that $E^s$ and $E^u$ are $A$-invariance (hence we have Condition (a)) and $E^s(\omega)\neq E^u(\omega)$ for all $\omega\in\Omega$. By Lemma~\ref{l:svd_smooth}, $s_n$ and $u_n$ are continuous in $\omega$ for all large $n$ which implies the continuity of $E^s$ and $E^u$. Hence, by compactness of $\Omega$, we have 
		\[
		\delta:=\inf_{\omega\in\Omega}d(E^s(\omega), E^u(\omega))>0,
		\] 
		which by Proposition~\ref{p:equiv_conditions_DS_dyna} implies Condition (b). Finally, we may follow the proof of Lemma~\ref{l:EuDominatesEs} to obtain Condition (c). In particular, we can set $\lambda=2$ and the choice of $N$ is independent of $\omega$ as it only depends on $\delta$ and $\mu$. 
		\end{proof}

\section{An Avalanche Principle for $\MM(2,\C)$-sequences}

In this section, we prove Theorem~\ref{t:APSingular}. Recall that by \eqref{eq:svg_decomp}, for any $Q\in\MM(2,\C)$, it holds that 
	$$
	Q=U(Q)\begin{pmatrix}\sigma_1(Q) &0\\ 0 &\sigma_2(Q)\end{pmatrix}V^*(Q),
	$$	
	where $V(Q)\cdot\infty=s(Q)$ and $U(Q)\cdot 0=u(Q)$. Let $E_1, E_2\in\MM(2,\C)$. We write $V(E_2)=(\vec v_1, \vec v_2)$, $U(E_1)=(\vec u_1, \vec u_2)$, where all vectors are column vectors. We also write
	\[
	V^*(E_2)U(E_1)=\begin{pmatrix}c_1&c_2\\ c_3&c_4\end{pmatrix}\in \mathrm{U}(2).
	\] 
	Note $V^*=V^{-1}$ which implies:
	\begin{equation}\label{eq:dist_su}
	|c_1|=|\det(\vec v_2, \vec u_1)|=d(V(E_2)\cdot\infty, U(E_1)\cdot 0)=d(s(E_2), u(E_1)).
	\end{equation}

	\begin{lemma}\label{l:NormToAngle}
		Let $E_1,E_2\in \MM(2,\C)$ satisfying 
		\begin{equation}\label{eq:ProdNormLarge}
\sigma_1(E_2E_1)>C^2 \max\left\{\sigma_1(E_1)\sigma_2(E_2),\sigma_1(E_2)\sigma_2(E_1)\right\}.
		\end{equation}
		Then it holds that 
		\begin{align}\label{eq:anglenorm}
		&cd\big(s(E_2),u(E_1)\big)<\frac{\sigma_1(E_2E_1)}{\sigma_1(E_2)\sigma_1(E_1)}<Cd\big(s(E_2),u(E_1)\big)\\
	\label{eq:anglenorm2}
		&\left|\frac{\sigma_1(E_2E_1)}{\sigma_1(E_2)\sigma_1(E_1)}-d(s(E_2),u(E_1))\right|<C\max\left\{\frac{\sigma_2(E_2)}{\sigma_1(E_2)},\ \frac{\sigma_2(E_1)}{\sigma_1(E_1)}\right\}
		\end{align}
	\end{lemma}
	\begin{proof}
	By singular value decomposition of $E_1$ and $E_2$, one has 
		$$
		E_2E_1=U(E_2)\begin{pmatrix}\sigma_1(E_2) &0\\ 0 &\sigma_2(E_2)\end{pmatrix}V^*(E_2)U(E_1)\begin{pmatrix}\sigma_1(E_1) &0\\ 0 &\sigma_2(E_1)\end{pmatrix}V^*(E_1),
		$$
	which implies that
		$$
		U^*(E_2)\cdot E_2E_1\cdot V(E_1)-\begin{pmatrix}c_1\sigma_1(E_2)\sigma_1(E_1) &0\\ 0 &0\end{pmatrix}=\begin{pmatrix}0 &c_2\sigma_1(E_2)\sigma_2(E_1)\\ c_3\sigma_2(E_2)\sigma_1(E_1) &c_4\sigma_2(E_2)\sigma_2(E_1)\end{pmatrix}.
		$$
		Since $\mathrm{U}(2)$-matrices preserve the operator norm, triangle inequality yields
		\begin{equation}\label{eq:AngleNorm2}
		\left|\sigma_1(E_2E_1)-c_1\sigma_1(E_2)\sigma_1(E_1)\right|<C\max\big\{\sigma_1(E_2)\sigma_2(E_1),\ \sigma_2(E_2)\sigma_1(E_1)\big\}.
		\end{equation}
		Combining \eqref{eq:AngleNorm2} with \eqref{eq:ProdNormLarge}, we obtain
		$$
		c|c_1|\sigma_1(E_2)\sigma_1(E_1)<\sigma_1(E_2E_1)<C|c_1|\sigma_1(E_2)\sigma_1(E_1),
		$$ 
	which together with \eqref{eq:dist_su} clearly implies \eqref{eq:anglenorm}. Now divide \eqref{eq:AngleNorm2} by $\sigma_1(E_2)\sigma_1(E_1)$ at both sides we obtain
		$$
		\left|\frac{\sigma_1(E_2E_1)}{\sigma_1(E_2)\sigma_1(E_1)}-d(s(E_2),u(E_1))\right|<C\max\left\{\frac{\sigma_2(E_2)}{\sigma_1(E_2)},\ \frac{\sigma_2(E_1)}{\sigma_1(E_1)}\right\}
		$$
		which is nothing other than \eqref{eq:anglenorm2}.
	\end{proof}

	\begin{corollary}\label{c:NormToAngle}
		
			Let $E_2, E_1\in\MM(2,\mathbb C)$ be such that   $\frac{\sigma_1(E_1)\sigma_1(E_2)}{\sigma_1(E_2E_1)}\le \mu^{1/4}$ and $\frac{\sigma_2(E_i)}{\sigma_1(E_i)}\le\mu^{-1}$ for $i=1,2$. Then it holds that
		\beq\label{eq:ap_us_1st_sep}
		d\big(s(E_2),u(E_1)\big)>c\mu^{-\frac14}.
		\eeq
	\end{corollary}
	\begin{proof}
	By the given conditions, we clearly have
		\begin{align*}
		\sigma_1\big(E_2E_1\big)\ge \sigma_1(E_2)\sigma_1(E_1)\mu^{-\frac14}
		>\mu^{\frac34}\max\left\{\sigma_1(E_1)\sigma_2(E_2), \sigma_2(E_1))\sigma_1(E_2)\right\}.
		\end{align*}
		Thus the condition of Lemma~\ref{l:NormToAngle} is satisfied which in turn implies
		$$
		d\big(s(E_2),u(E_1)\big)>c\frac{\sigma_1(E_2E_1)}{\sigma_1(E_2)\sigma_1(E_1)}\ge c\mu^{-\frac14}.
		$$
	\end{proof}

	\begin{lemma}\label{l:NormDirectionControl1}
		Let $E_2, E_1\in\MM(2,\mathbb C)$ such that  $d(s(E_2),u(E_1))>c\mu^{-\frac14}$ and $\frac{\sigma_2(E_i)}{\sigma_1(E_i)}\le\mu^{-1}$ for $i=1,2$. Let $E=E_2E_1$. Then it holds that
		\begin{align}
			\label{eq:NormControl1}
			&\frac{\sigma_2(E)}{\sigma_1(E)}<C\frac{\sigma_2(E_1)\sigma_2(E_2)}{\sigma_1(E_1)\sigma_1(E_2)}\cdot d(s(E_2),u(E_1))^{-2}<C\mu^{-\frac32},\\
			\label{eq:SControl1}
			&d\big(s(E_1),s(E)\big)<C\frac{\sigma_2(E_1)}{\sigma_1(E_1)}d(s(E_2),u(E_1))^{-1}<C\mu^{-\frac34},\\
			\label{eq:UControl1}
			&d(u(E_2),u(E))<C\frac{\sigma_2(E_2)}{\sigma_1(E_2)}d(s(E_2),u(E_1))^{-1}<C\mu^{-\frac34}.
		\end{align}
	\end{lemma}
	\begin{proof}
		Let  $G:=V^*(E_2)U(E_1)=\left(\begin{smallmatrix}c_1&c_2\\ c_3&c_4\end{smallmatrix}\right)\in \mathrm{U}(2)$. Recall by \eqref{eq:dist_su}, we have $|c_1|=d(s(E_2), u(E_1))\ge c\mu^{-\frac14}$.  Define $D$ to be
		$$
		D=
		\begin{pmatrix}
			\sigma_1(E_2)&0\\0&\sigma_2(E_2)\end{pmatrix}G\begin{pmatrix}\sigma_1(E_1)&0\\0&\sigma_2(E_1)\
		\end{pmatrix}
		=
		\begin{pmatrix}c_1\sigma_1(E_1)\sigma_1(E_2)  &c_2\sigma_1(E_2)\sigma_2(E_1)\\ c_3\sigma_2(E_2)\sigma_1(E_1) &c_4\sigma_2(E_1)\sigma_2(E_2)\end{pmatrix}.
		$$
		It is clear that $\sigma_1(D)\ge c|c_1|\sigma_1(E_1)\sigma_1(E_2) $, which implies
		\begin{align}\label{eq:large_D_norm}
		\nonumber \frac{\sigma_2(E)}{\sigma_1(E)}&=	\frac{\sigma_2(D)}{\sigma_1(D)}=\frac{|\det(D)|}{\sigma_1^2(D)}\\
		&\le C|c_1|^{-2}\frac{\sigma_2(E_1)\sigma_2(E_2)}{\sigma_1(E_1)\sigma_1(E_2)}\\
		\nonumber &<C\mu^{-3/2},
		\end{align}
		which takes care of \eqref{eq:NormControl1}. Let $\vec e=\binom{0}{1}$. By the form of $D$, it is clearly that
		$$
		\|D\vec e\|\le |c_2|\sigma_1(E_2)\sigma_2(E_1)+|c_4|\sigma_2(E_1)\sigma_2(E_2).
		$$
		Let $\gamma=d\big(\infty, s(D)\big)$.  It is clear that $s(DV^*(E_1))=V(E_1)\cdot s(D)$ which implies that
		\begin{align*}
			d\big(s(E_1),s(E)\big)&=d\big(s(E_1),s(DV^*(E_1))\big)\\
			&=d\big(V(E_1)\cdot\infty, V(E_1)s(D)\big)\\
			&=d\big(\infty, s(D)\big)\\
			&=\gamma,
		\end{align*}
		where the third equality follows from the fact that $\mathrm{U}(2)$-matrices preserve the distance $d$. Writing $\vec e=d_1\vec s(D)+d_2\vec s^{\perp}(D)$. Then it is clear that $|d_2|=d(\vec e, \vec s(D))=d(\infty, s(D))=\gamma$.
		
		If $\gamma\le \frac{\sigma_2(E_1)}{\sigma_1(E_1)}$, then \eqref{eq:SControl1} is automatically true as we clearly have $d\big(s(E_2),u(E_1)\big)^{-1}>c$. So we only need to deal with the case where $\gamma>\frac{\sigma_2(E_1)}{\sigma_1(E_1)}$, which together with the proof of  \eqref{eq:large_D_norm} implies that
		\begin{align*}
			\|d_2D\vec s^{\perp}(D)\|&>\frac{\sigma_2(E_1)}{\sigma_1(E_1)}\sigma_1(D)\\
			&>\frac{\sigma_2(E_1)}{\sigma_1(E_1)}|c_1|^{2}\frac{\sigma_1(E_1)\sigma_1(E_2)}{\sigma_2(E_1)\sigma_2(E_2)}\sigma_2(D)\\
			&=|c_1|^{2}\frac{\sigma_1(E_2)}{\sigma_2(E_2)}\sigma_2(D)\\
			&>c\mu^{\frac12}\sigma_2(D)\\
			&>C|d_1D\vec s(D)|,
		\end{align*}
		and hence, $\|D\vec e\|=\|d_1D\vec s(D)+d_2D\vec s^{\perp}(D)\|\ge \|d_2D\vec s^{\perp}(D)\|-\|d_1D\vec s(D\|\ge c\|d_2D\vec s^{\perp}(D)\|$. Combining all the estimates above, we obtain
		\begin{align*}
			|\gamma|&=|d_2|<\frac{C}{\sigma_1(D)}\|D\vec e\|\\
			&<\frac{C}{|c_1|\sigma_1(E_1)\sigma_1(E_2)}(|c_2|\sigma_1(E_2)\sigma_2(E_1)+|c_4|\sigma_2(E_1)\sigma_2(E_2))\\
			&<C\frac{\sigma_2(E_1)}{\sigma_1(E_1)}d(s(E_2),u(E_1))^{-1}
		\end{align*}
		which is nothing other than  \eqref{eq:SControl1}. Running the same argument above with $E^{*}$, $D^{*}$, $s(E^*)=u^\perp(E)$, $s(D^*)=u(D)^\perp$, and $s(E_2^*)=u(E^*_2)^\perp$, one obtains \eqref{eq:UControl1}.
	\end{proof}
	
	The following lemma push the estimates in Lemma~\ref{l:NormDirectionControl1} to all $n\ge 2$. Recall that $s_n(j)=s(B_n(j))$ and $u_n(j)=s(B^*_n(j-n))^\perp$.
	\begin{lemma}\label{l:NormDirectionControln}
		Let $B:\Z\to\MM(2,\C)$ be as in Theorem~\ref{t:APSingular}. Then it holds for each $j\in\Z$ and each $n\ge 2$ that 
		\begin{align}
			\label{eq:NormControl} \frac{\sigma_2(B_n(j))}{\sigma_1(B_n(j))}&\le C\mu^{-\frac{n+1}{2}},\\
			\label{eq:SControl}d\big(s_n(j),s_{n-1}(j)\big)&<C\mu^{-\frac{n-1}2},\\
			\label{eq:UControl}d\big(u_n(j),u_{n-1}(j)\big)&<C\mu^{-\frac{n-1}2}.
		\end{align}
	\end{lemma}
	\begin{proof}
		We proceed by induction on $n$. Note for the case $n=2$, \eqref{eq:NormControl} and \eqref{eq:SControl} follow from \eqref{eq:NormControl1} and \eqref{eq:SControl1} by setting $E_1=B(j)$ and $E_2=B(j+1)$ and the fact $\mu^{-\frac34}<\mu^{-\frac12}$. Similarly, \eqref{eq:UControl} follows from \eqref{eq:UControl1} if we set $E_1=B(j-2)$ and $E_2=B(j-1)$. 
		
		Assuming that \eqref{eq:NormControl}-\eqref{eq:SControl} hold true for all $n=2,\ldots, k$ and all $j\in\Z$.  Then we want to move to the case $n=k+1$. First, it holds that
		\begin{align}
			\nonumber d\big(u_{k}(j+k),u_1(j+k)\big)&\le \sum^{k-1}_{l=1}d\big(u_{l+1}(j+k),u_l(j+k)\big)\\
						\label{eq:AngleControl0}&\le C\mu^{-\frac34}+C\sum^{k-1}_{l=2}\mu^{-\frac{l}2}\\
\nonumber &\le C\mu^{-\frac34}.
		\end{align}
Also, applying Corollary~\ref{c:NormToAngle} to $E_2=B(j)$ and $E_1=B(j-1)$, we obtain for all $j\in\Z$ that
\[
d(s_1(j), u_1(j))=d(s(B(j)), u(B(j-1)))\ge c\mu^{-\frac14}.
\]
Combining the two estimates above, we obtain that
		\begin{align}
			\nonumber d\big(s_1(j+k), u_k(j+k)\big)
			&\ge d\big(s_1(j+k),u_1(j+k)\big)-d\big(u_1(j+k),u_k(j+k)\big)\\
	\label{eq:AngleControl}			&\ge c\mu^{-\frac14}-C\mu^{-\frac34}\\
		\nonumber &\ge c\mu^{-\frac14}.
		\end{align}
		
				Thus, we may apply Lemma~\ref{l:NormDirectionControl1} with $E_1=B_k(j)$ and $E_2=B(j+k)$ to obtain the following two estimates. First, we get that
		\begin{align*}
			\frac{\sigma_2(B_{k+1}(j))}{\sigma_1(B_{k+1}(j))}
			&\le\frac{\sigma_2(B(j+k))\sigma_2(B_{k}(j))}{\sigma_1(B(j+k))\sigma_1(B_{k}(j))}d\big(s_1(j+k), u_k(j+k)\big)^{-1}\\
			&<C\mu^{-1}\mu^{-\frac{k+1}{2}}\mu^{\frac14}\\
			&<C\mu^{-\frac{k+2}{2}}
		\end{align*}
		which takes care of \eqref{eq:NormControl} for $n=k+1$. Next, we obtain
		\begin{align}
			\nonumber d(s_{k+1}(j), s_k(j))&=d\big(s[B(k+j)B_k(j)], s(B_k(j))\big)\\
			\nonumber &<C\frac{\sigma_2(B_k(j))}{\sigma_1(B_k(j))}\cdot d(s_1(j+k), u_k(j+k))^{-2}\\
			\nonumber &<C\mu^{-\frac{k+1}2}\mu^{\frac12}\\
			&=C\mu^{-\frac k2},
		\end{align}
		which clearly takes care of the \eqref{eq:SControl} for $n=k+1$.

		Similarly, by the same argument of \eqref{eq:AngleControl0}, it holds that $d\big(s_k(j-k-1), s_1(j-k-1)\big)\le C\mu^{-\frac34}$. Together with Corollary~\ref{c:NormToAngle}, we then obtain
		\begin{align}
			\nonumber &d\big(s_k(j-k-1), u_1(j-k-2)\big)\\
			\nonumber \ge & d\big(s_1(j-k-1),u_1(j-k-2)\big)-d\big(s_k(j-k-1),s_1(j-k-1)\big)\\
			\label{eq:AngleControl2}\ge & c\mu^{-\frac14}-C\mu^{-\frac34}\\
			\nonumber \ge & c\mu^{-\frac14}.
		\end{align}
Now combining \eqref{eq:AngleControl2} and \eqref{eq:UControl1} with $E_1=B(k-j-2)$ and $E_2=B_k(j-k-1)$, we obtain
		\begin{align}
			\nonumber d\big(u_{k+1}(j),u_k(j)\big)&=d\big(u[B_k(j-k-1)B(j-k-2)], u(B_k(j-k-1))\big)\\
			\nonumber &<C\frac{\sigma_2(B_k(j-k-1))}{\sigma_1(B_k(j-k-1))}\cdot d\big(s_k(j-k-1),u(j-k-2)\big)^{-2}\\
			\nonumber &<C\mu^{-\frac{k+1}2}\mu^{\frac12}\\
			&=C\mu^{-\frac{k}{2}},
		\end{align}
		which takes of  \eqref{eq:UControl} for step $n=k+1$, concluding the proof.
	\end{proof}
	
	Now, we are ready to prove Theorem~\ref{t:APSingular}.
	\begin{proof}[Proof of Theorem~\ref{t:APSingular}]
		First, we show that $B:\Z\to\MM(2,\C)$ had dominated splitting. Setting $k=n$ in \eqref{eq:AngleControl}, we obtain $d(s_1(j+n), u_n(j+n))\ge c\mu^{-\frac14}$. By the proof of Lemma~\ref{l:NormDirectionControl1}, we have
		\begin{align*}
		\sigma_1(B_{n+1}(j))&\ge cd(s_1(j+n), u_n(j+n))\sigma_1(B_n(j))\sigma_1(B(j+n))\\
		&\ge c\mu^{-\frac14}\sigma_1(B_n(j))\sigma_1(B(j+n)),
		\end{align*}
		Hence, we have:
		\begin{equation}\label{eq:svg1}
		\frac{\sigma_1(B_{n}(j))}{\sigma_1(B_{n+1}(j))}\le \frac{\mu^{\frac14}}{c\sigma_1(B(j+n))}\le C\mu^{\frac14},
		\end{equation}
		which together with \eqref{eq:NormControl} implies 
		\begin{align}\label{eq:fi1}
		\frac{\sigma_2(B_n(j))}{\sigma_1(B_{n+1}(j))}\le C\mu^{\frac14}\frac{\sigma_2(B_n(j))}{\sigma_1(B_n(j))}\le C\mu^{-\frac14}\mu^{-\frac{n}{2}}.
		\end{align}
	   Similarly, setting $k=n$ and replacing $j-k-2$ by $j$ in \eqref{eq:AngleControl2} yield
	   \[
	   d(s_n(j+1), u_1(j))>c\mu^{-\frac14},
	   \]
	   which implies
	   \begin{align*}
	   	\sigma_1(B_{n+1}(j))&\ge cd(s_n(j+1), u_1(j+n))\sigma_1(B_n(j+1))\sigma_1(B(j))\\
	   &\ge c\mu^{-\frac14}\sigma_1(B_n(j+1))\sigma_1(B(j)).
	   \end{align*}
		Hence, we obtain 
			\begin{equation}\label{eq:svg2}
			\frac{\sigma_1(B_{n}(j+1))}{\sigma_1(B_{n+1}(j))}\le \frac{\mu^{\frac14}}{c\sigma_1(B(j))}\le C\mu^{\frac14}
		\end{equation}
		and hence by \eqref{eq:NormControl},
			\begin{align}\label{eq:fi2}
			\frac{\sigma_2(B_n(j+1))}{\sigma_1(B_{n+1}(j))}\le C\mu^{\frac14}\frac{\sigma_2(B_n(j+1))}{\sigma_1(B_n(j+1))}\le C\mu^{-\frac14}\mu^{-\frac{n}{2}}.
		\end{align}
		It is clear that \eqref{eq:svg1} and \eqref{eq:svg1} imply $\SVG$ for $B$ while \eqref{eq:fi1} and \eqref{eq:fi2} imply $\FI$ for $B$. By Theorem~\ref{t:main}, $B$ has dominated splitting.

		For the proof of \eqref{eq:AP}, we first note it holds for all $j\in\Z$ and $n\in\Z_+$ that
		\beq\label{eq:AP1}
		\log\sigma_1(B_n(j))=\log\sigma_1(B(j+n-1))+\log\sigma_1(B_{n-1}(j))+\log\frac{\sigma_1\big(B(j+n-1)B_{n-1}(j)\big)}{\sigma_1(B(j+n-1))\sigma_1(B_{n-1}(j))}.
		\eeq
		We may apply \eqref{eq:AP1} to $\log\sigma_1(B_{n-1}(j))$ and rewrite \eqref{eq:AP1} as 
		\beq\label{eq:AP2}
		\log\sigma_1(B_n(j))=\log\sigma_1(B_{n-2}(j))+\sum^{n-1}_{k=n-2}\log\sigma_1(B(j+k))+\sum^{n-1}_{k=n-2}
		\log\frac{\sigma_1(B(j+k)B_{k}(j))}{\sigma_1(B(j+k))\sigma_1(B_{k}(j))}.
		\eeq
		Apply this process repeatedly to $B_n(j), B_{n-1}(j),\ldots, B_2(j)$, we then obtain
		\beq\label{eq:PreAP}
		\log\sigma_1(B_n(j))=\sum^{n-1}_{k=0}\log\sigma_1(B(j+k))+\sum^{n-1}_{k=1}\log\frac{\sigma_1(B(j+k)B_{k}(j))}{\sigma_1(B(j+k))\cdot\sigma_1(B_{k}(j))}.
		\eeq
		Now for each $k\ge1$, by the proof of Corollary~\ref{c:NormToAngle}, condition of Lemma~\ref{l:NormToAngle} is satisfied for the pair $B(j+k)$ and $B(j+k-1)$. Similarly, \eqref{eq:AngleControl} implies the condition is satisfied for $B(j+k)$ and $B_k(j)$ as well. Thus, applying \eqref{eq:anglenorm2} to both pairs, we obtain
		$$
		\left|\frac{\sigma_1(B(j+k)B_{k}(j))}{\sigma_1(B(j+k))\cdot\sigma_1(B_{k}(j))}-d\big(s_1(j+k),u_k(j+k)\big)\right|<C\mu^{-1}
		$$
		and 
		$$
		\left|\frac{\sigma_1(B(j+k)B(j+k-1))}{\sigma_1(B(j+k))\cdot\sigma_1(B(j+k-1))}-d\big(s_1(j+k),u_1(j+k)\big)\right|<C\mu^{-1}.
		$$
		On the other hand, by \eqref{eq:AngleControl0} it holds that
		\begin{align*}
			&\left| (d\big(s_1(j+k),u_k(j+k)\big)- d\big(s_1(j+k),u_1(j+k)\big)\right|\\
			&\le d\big(u_k(j+k), u_1(j+k)\big)\\
			&\le C\mu^{-\frac34},
		\end{align*}
		where the last inequality follows from \eqref{eq:AngleControl0}. Combine the three inequalities above, we then obtain
		$$
		\left|\frac{\sigma_1(B(j+k)B_{k}(j))}{\sigma_1(B(j+k))\cdot\sigma_1(B_{k}(j))}-\frac{\sigma_1(B(j+k)B(j+k-1))}{\sigma_1(B(j+k))\cdot\sigma_1(B(j+k-1))}\right|<C\mu^{-\frac34}.
		$$
		Apply \eqref{eq:anglenorm} and Corollary~\ref{c:NormToAngle} to $B(j+k)$ and $B(j+k-1)$, we obtain
		$$
		\frac{\sigma_1(B(j+k)B(j+k-1))}{\sigma_1(B(j+k))\cdot\sigma_1(B(j+k-1))}>cd\big(s_1(j+k),u_1(j+k)\big)>c\mu^{-\frac14}.
		$$
		It is straightforward calculus type of estimate that 
		$$
		|\log a -\log b|<C\left|\frac1b(a-b)\right| \mbox{ when } b>C|a-b|. 
		$$
		Thus for each $k\ge1$, the inequality above implies that
		\begin{align*}
			&\left|\log\frac{\sigma_1(B(j+k)B_{k}(j))}{\sigma_1(B(j+k))\cdot\sigma_1(B_{k}(j))}-\log\frac{\sigma_1(B(j+k)B(j+k-1))}{\sigma_1(B(j+k))\cdot\sigma_1(B(j+k-1))}\right|\\
			&\le C\frac{\sigma_1(B(j+k))\cdot\sigma_1(B(j+k-1))}{\sigma_1(B(j+k)B(j+k-1))}\cdot\left|\frac{\sigma_1(B(j+k)B_{k}(j))}{\sigma_1(B(j+k))\cdot\sigma_1(B_{k}(j))}-\frac{\sigma_1(B(j+k)B(j+k-1))}{\sigma_1(B(j+k))\cdot\sigma_1(B(j+k-1))}\right|\\
			&\le
			C\mu^{\frac14}\mu^{-\frac34}\\
			&=C\mu^{-\frac12}.
		\end{align*}
		Combine \eqref{eq:PreAP} and the estimate above, we then obtain
		\begin{align*}
			&\left|\log\sigma_1(B_n(j))-\sum^{n-1}_{k=0}\log\sigma_1(B(j+k))-\sum^{n-1}_{k=1}\log\frac{\sigma_1(B(j+k)B(j+k-1))}{\sigma_1(B(j+k))\cdot\sigma_1(B(j+k-1))}\right|\\
			&=\left|\sum^{n-1}_{k=1}\log\frac{\sigma_1(B(j+k)B_{k}(j))}{\sigma_1(B(j+k))\cdot\sigma_1(B_{k}(j))}-\sum^{n-1}_{k=1}\log\frac{\sigma_1(B(j+k)B(j+k-1))}{\sigma_1(B(j+k))\cdot\sigma_1(B(j+k-1))}\right|\\
			&\le \sum^{n-1}_{k=1}\left|\log\frac{\sigma_1(B(j+k)B_{k}(j))}{\sigma_1(B(j+k))\cdot\sigma_1(B_{k}(j))}-\log\frac{\sigma_1(B(j+k)B(j+k-1))}{\sigma_1(B(j+k))\cdot\sigma_1(B(j+k-1))}\right|\\
			&\le C(n-1)\mu^{-\frac12}\\
			&\le Cn\mu^{-\frac12}.
		\end{align*}
		A direct computation shows that the first line in the estimate above is nothing other than 
		$$
		\left|\log\|B_n(j)\|+\sum_{k=1}^{n-2}\log\|B(j+k)\|-\sum_{k=0}^{n-2}\log\|B(j+k+1)B(j+k)\|\right|,
		$$
		concluding the proof.
	\end{proof}

\end{document}